\documentclass[11pt,twoside,english,reqno,a4paper,final]{amsart}
\usepackage[utf8]{inputenc}
\usepackage[english]{babel}
\usepackage{cancel}
\usepackage[lined,boxed]{algorithm2e}
\usepackage{float} 
\usepackage{hyperref}
\usepackage[notref,notcite]{showkeys}
\usepackage{fullpage}
\usepackage{enumitem}
\usepackage{tikz}
\usepackage{tikz-3dplot}
\usetikzlibrary{3d}
\usepackage{amsmath}

\usepackage{listings,graphicx,amsmath,varioref,amscd,amssymb,color,xcolor,bm,amsthm,amsfonts,graphics,lineno,float,comment}
\usepackage{mathrsfs} 

\newcommand{\s}{{\sigma}}
\newcommand{\R}{\mathbb R}
\newcommand{\Div}{\hbox{div\,}}
\DeclareMathOperator{\sectional}{\kappa}

\theoremstyle{plain}
\newtheorem{theorem}{Theorem}[section]
\newtheorem{proposition}[theorem]{Proposition}

\newtheorem{lemma}[theorem]{Lemma}

\theoremstyle{definition}
\newtheorem{definition}[theorem]{Definition}

\newtheorem{remark}[theorem]{Remark}

\theoremstyle{remark}

\numberwithin{equation}{section}
\author{Giulia Meglioli}

\address{\hbox{\parbox{5.7in}{\medskip \noindent{Giulia Meglioli, \\Fakult\"at f\"ur Mathematik, \\Universit\"at Bielefeld, \\33501, Bielefeld, Germany \\ [3pt] \emph{E-mail address: }{\tt gmeglioli@math.uni-bielefeld.de}}}}}

\author{Francescantonio Oliva}

\address{\hbox{\parbox{5.7in}{\medskip \noindent{Francescantonio Oliva, \\Dipartimento di Scienze di Base e Applicate per l'Ingegneria, \\"Sapienza" Universit\'a di Roma, \\Via Scarpa 16, 00161 Rome, Italy \\ [3pt] \emph{E-mail address: }{\tt francescantonio.oliva@uniroma1.it}}}}}

\author{Francesco Petitta}

\address{\hbox{\parbox{5.7in}{\medskip \noindent{Francesco Petitta, \\Dipartimento di Scienze di Base e Applicate per l'Ingegneria, \\"Sapienza" Universit\'a di Roma, \\Via Scarpa 16, 00161 Rome, Italy \\ [3pt] \emph{E-mail address: }{\tt francesco.petitta@uniroma1.it}}}}}

\begin{document}

\title[]{Global existence for a Leibenson type equation with reaction on Riemannian manifolds}



\maketitle

\begin{abstract}
We show a global existence result for a doubly nonlinear porous medium type equation of the form  
$$u_t = \Delta_p u^m +\, u^q$$ 
on a complete and non-compact Riemannian manifold $M$ of infinite volume. Here, for $1<p<N$, we assume   $m(p-1)\ge1$, $m>1$ and $q>m(p-1)$. In particular, under the assumptions that $M$ supports the Sobolev inequality, we prove that a solution for such a problem exists globally in time provided $q>m(p-1)+\frac pN$ and the initial datum is small enough; namely, we establish an explicit bound on the $L^\infty$ norm of the solution at all positive times, in terms of the $L^1$ norm of the data. Under the additional assumption that a Poincar\'e-type inequality also holds in $M$, we can establish the same result in the larger interval, i.e. $q>m(p-1)$. This result has no Euclidean counterpart, as it differs entirely from the case of a bounded Euclidean domain due to the fact that $M$ is non-compact and has infinite measure.
\end{abstract}

\section{Introduction}
Let $ (M,g)$ be a complete non-compact Riemannian manifold of infinite volume and dimension $N>1$. We investigate existence of non-negative global in time solutions to the quasilinear parabolic problem
\begin{equation}\label{problema}
\begin{cases}
u_t = \Delta_p u^m +\, u^q & \text{in}\,\, M \times (0,T) \\
u =u_0 &\text{in}\,\, M\times \{0\}\,,
\end{cases}
\end{equation}
where $T\in(0,+\infty]$, $\Delta_p$ is the $p$-Laplacian defined as 
$$\Delta_p v:=\Div\left(\left|\nabla v\right|^{p-2}\nabla v\right),$$ and 
\begin{equation}\label{hp}
q>m(p-1)\ge1,\quad\quad 1<p<N,  \quad\quad m>1.
\end{equation}
Here and throughout the rest of the paper $\Div$ and $\nabla$ stands, respectively,  for  the divergence and the gradient defined on the  manifold $M$ with respect to the metric $g$. We will also always assume   that the initial datum $u_0$ is non-negative.
We consider the usual Lebesgue space of measurable functions $f$ such that $|f|^q$ is integrable w.r.t. the Riemannian measure $\mu$ induced by $g$, i.e., 
$$L^q(M)=\left\{v:M\to\R\,\, \text{measurable}\,\,  ,   \,\, \|v\|_{L^q(M)}:=\left(\int_{M}|v|^q\,d\mu\right)^{1/q}<+\infty\right\}.$$ 
Furthermore, we make the following assumption on $M$:  we assume the validity of the Sobolev inequality, i.e.
\begin{equation}\label{S}
\|v\|_{L^{p^*}(M)} \le \frac{1}{\mathcal{S}_p} \|\nabla v\|_{L^p(M)}\quad \text{for any}\,\,\, v\in C_c^{\infty}(M),
\end{equation}
where $\mathcal{S}_p>0$ is a numerical constants and $p^*:=\frac{pN}{N-p}$.   The validity of \eqref{S} puts constraints on $M$, though  it is e.g. well known that, on \it Cartan-Hadamard manifolds\rm, namely complete and simply connected manifolds that have everywhere non-positive sectional curvature, \eqref{S} always holds. This covers a wide class of manifolds, including  the hyperbolic space $\mathbb{H}^N$ and, of course, also the Euclidean space $\R^N$.
In some specific cases, we will furthermore assume the validity of the Poincar\'{e} inequality on $M$:
\begin{equation}\label{P}
 \|v\|_{L^p(M)} \le \frac{1}{\mathcal{C}_p} \|\nabla v\|_{L^p(M)} \quad \text{for any}\,\,\, v\in C_c^{\infty}(M),
\end{equation}
for some numerical constant $\mathcal{C}_p >0$. This assumption certainly   puts an extra constraint on the manifold: \eqref{P} holds when $M$ is Cartan-Hadamard and, besides, $\sectional\le -c<0$ everywhere, see e.g. \cite{Grig, Grig3}. Here, $\sectional$ denotes the sectional curvature of $M$. Thus, in particular, note that both \eqref{S} and \eqref{P} hold when $M$ is Cartan-Hadamard and $\sectional\le -c<0$ everywhere, a case that differs significantly from the Euclidean space, where the sectional curvature is zero. As already mentioned, this scenario encompasses a broad class of manifolds, including the fundamental example of the hyperbolic space $\mathbb{H}^n$, namely that Cartan-Hadamard manifold with constant sectional curvatures $-1$ (or similarly, the case where $\sectional=-c$ everywhere, for some fixed $c>0$).
\medskip

\subsection{Related results in the literature}
The equation in problem \eqref{problema}, when the power-like source term is neglected, i.e.,
\begin{equation}\label{leibenson}
u_t = \Delta_p u^m,
\end{equation}
is referred to as the {\it Leibenson} equation or a {\it doubly non-linear parabolic} equation and it is used to describe filtration of
turbulent compressible fluids through a porous medium. Clearly, the case when \(m=1\) corresponds to the evolutionary \(p\)-Laplace equation \(u_t = \Delta_p u\) whereas, when \(p=2\), \eqref{leibenson} becomes the porous medium equation, \(u_t=\Delta(u^m)\). In the special case \(m=1\) and \(p=2\), it reduces to the heat equation. Equation \eqref{leibenson} posed in \(\mathbb{R}^N\) admits spherically symmetric self-similar solutions known as {\it Barenblatt solutions}, see \cite{Bar}. If \(m(p-1)>1\), the Barenblatt solution \(u(x,t)\) has compact support for all \(t>0\), meaning that \(u\) has finite propagation speed; this range is also called {\it slow diffusion} regime. On the other hand, if \(m(p-1)\le1\), then \(u(x,t)\) is positive for all \(x\in\mathbb{R}^N\) and \(t>0\), implying infinite propagation speed (i.e. {\it fast diffusion}). The case \(m(p-1) =1\) is also called the pseudo-linear regime. 

\noindent Problem \eqref{leibenson} has been extensively studied also on Riemannian manifolds, see \cite{BG, GrigSur1, GrigSur2, GrigSur3, GM2, GMPrm, Pu1}. Specifically, in \cite{GrigSur1}, the authors prove that, similarly to the Euclidean setting, solutions to \eqref{leibenson} on arbitrary Riemannian manifolds exhibit a finite rate of propagation for \(m(p-1)>1\) (see also \cite{GrigSur2} for related results). This behavior is not totally surprising, in fact, it is well-known that negative curvature accelerates diffusions. Already from solutions to the heat equation, it has been showed that, the standard deviation of a Brownian particle acts linearly in time on the hyperbolic space $\mathbb H^N$, whereas, in the Euclidean setting, it is proportional to the square root of the time variable, see e.g. \cite{Grig3}.
\medskip

Let us now briefly comment on the reaction term \(u^q\), with \(q>1\). In the Euclidean setting (\(M = \mathbb{R}^N\)), if \(m = 1\) and \(p=2\) in \eqref{problema}, it is well-known that if \(q \leq 1 + \frac{2}{N}\), the solution will blow up in finite time for any \(u_0 \neq 0\). However, if \(q > 1 + \frac{2}{N}\) and the initial datum \(u_0\) is bounded and sufficiently small (in a suitable Lebesgue norm), global existence holds. This occurrence is known as the \textit{Fujita phenomenon} (see \cite{F, H}). For further results related to this problem we refer the reader to  \cite{CFG, DL, FI, KST, L, Poho, Q, S, Vaz1, W, Y} and references therein.

\noindent A similar result also holds if \(m > 1\) and \(p=2\), or \(m = 1\) and \(p>1\). More precisely, the results in \cite{SGKM} show that the solution to \eqref{problema} with \(p=2\) and $M=\R^N$ blows up for any \(q \leq m + \frac{2}{N}\) when \(u_0 \neq 0\), and that, on the other hand, a global solution exists if \(q > m + \frac{2}{N}\) with  a compactly supported, sufficiently small initial datum. For related results, see also \cite{Meg, MP1, MP2, MP3, MQV}. Similarly, the case of problem \eqref{problema} when \(M = \mathbb{R}^N\) and \(m = 1\), i.e. when the diffusion is driven by the $p$-Laplacian, has attracted a lot of attention, see e.g. \cite{Gal1, Gal2, Gal3, Mitidieri, Mitidieri2} and references therein. In particular, in \cite{Mitidieri2}, non-existence of nontrivial weak solutions is proved for any \(\max\{1,p-1\} < q \le (p-1) + \frac{p}{N}\) and \(p > \frac{2N}{N+1}\), while a global-in-time solution exists for any \(q > (p-1) + \frac{p}{N}\) provided that the initial datum is compactly supported and sufficiently small. The general equation in problem \eqref{problema} on the Euclidean space has been addressed in \cite{MT1, MTS, MTS2}.

\noindent On Riemannian manifolds, the situation can be significantly different from the Euclidean setting, especially in the case of negative curvature. In fact, when dealing with the case of the \(N\)-dimensional hyperbolic space, \(M = \mathbb{H}^N\), it is known that, problem \eqref{problema} with \(p=2\) and \(m=1\), for all \(q>1\) and a sufficiently small non-negative datum, admits a global-in-time solution, see \cite{BPT, Pu3, Sun1, WY, WY2, Z}, meaning that the Fujita phenomenon does not appear. Furthermore, in \cite{GMPFuj}, the authors provide a class of reaction terms $f(u)$ which replace $u^q$ in \eqref{problema} and restore the Fujita phenomenon on complete, non-compact, stochastically complete Riemannian manifold with \(\lambda_1(M)>0\); here \(\lambda_1(M):=\inf \operatorname{spec}(-\Delta)\), is the infimum of the $L^2$-spectrum of the  Laplace-Beltrami operator on $M$.

The porous medium equation with a power-like reaction on complete, non-compact, Riemannian manifolds of infinite volume, i.e., problem \eqref{problema} with \(p=2\), has been investigated in \cite{GMeP1, GMeP2, GMuP}. The authors in \cite{GMeP1, GMeP2} show global existence of solutions for small enough initial datum when \(q > m + \frac{2}{N}\), under the only assumption that the Sobolev inequality holds on \(M\) (this also covers the case of \(\mathbb{R}^N\)). Furthermore, they show that, on a smaller class of manifolds, i.e., when both the Sobolev and the Poincaré inequalities hold, global existence in $(0,T)$ for all $T>0$ follows either if \(q < m\) or \(q > m\). Finally, in \cite{GMuP}, the authors show that, on Cartan-Hadamard manifolds which satisfy proper bounds on the sectional and Ricci curvatures, for all $q<m$ it is possible to construct a subsolution to problem \eqref{problema} with \(p=2\) which blows-up as $t\to+\infty$. 

Concerning problem \eqref{problema} with \(m = 1\), we refer the reader to \cite{GMPplap, MMP, MeMoPu} and references therein. In particular, in \cite{GMPplap}, the author shows that when the Sobolev inequality \eqref{S} holds, \(q > (p-1) + \frac{p}{N}\) and \(\frac{2N}{N+1} < p < N\), a global solution exists for small enough datum. They also provide a smoothing estimate of the type \(L^1 \to L^\infty\). Moreover, if both the Sobolev and the Poincaré inequalities hold, global existence of solutions is proved in the larger interval \(q > p-1\) and \(\frac{2N}{N+1} < p < N\), provided that \(u_0\) is sufficiently small.

\subsection{Outline of our results}

In this paper, we prove the following results. Assume that the bounds in \eqref{hp} hold.
\begin{itemize}
\item[$(a)$] Assume that the Sobolev inequality \eqref{S} holds and besides that 
\begin{equation}\label{opt}
q>m(p-1)+\frac pN.\end{equation}
If $u_0\in {L}^{s}(M)\cap {L}^1(M)$ is sufficiently small, with $s>[q-m(p-1)]\frac Np$, then a global solution exists, see Theorem \ref{teo1}.
\item[$(b)$] Now suppose that both the Sobolev inequality \eqref{S} and the Poincar\'e inequality \eqref{P} hold without any further restriction on the exponent $q$.
If $u_0\in { L}^{s}(M)\cap { L}^{q\frac Np}(M)$ is sufficiently small, with $s>\max\left\{(q-m(p-1))\frac Np, 1\right\}$, then a global solution exists, see Theorem \ref{teo2}.
\end{itemize}
Observe that the situation in $(b)$ has of course no Euclidean analogue, as it is completely different from the case of a bounded Euclidean domain since $M$ is non-compact and of infinite measure. 

The result we prove in Theorem \ref{teo1} is sharp in the sense that below the exponent given in \eqref{opt} a non-existence result holds for every initial data, see \cite{Sun2}. More precisely, the Riemannian manifolds considered in Theorem \ref{teo1} are those where only inequality \eqref{S} holds, therefore also the euclidean space $\R^N$ is considered. For this reason, one can not expect global existence in the range of $q$ below \eqref{opt}. Non-existence of global solutions to more general non-linearities in the diffusion has been recently showed in \cite{mont}.

\subsection{A short overview of the applied techniques}
In order to prove the latter we adapt the methods exploited in  \cite{MT1}, \cite{GMeP1} and \cite{GMPplap}.

We first  obtain Caccioppoli-type estimates, which provide crucial a priori energy bounds for the solutions. Then, we apply these estimates in a Moser iteration method to derive a global $L^\infty$- 
  estimate, ensuring uniform control over the solution.   Finally, the existence of a solution is obtained through standard a priori estimates, combined with a passage to the limit in suitable approximating problems. By carefully constructing these approximations and ensuring their convergence, we rigorously justify the existence of a weak solution to problem \eqref{problema}.

\medskip 
The paper is organized as follows. The main results are stated in Section \ref{statement} together with the definition of solution to problem \eqref{problema}. Section \ref{aux} is dedicated to proving the Caccioppoli-type a priori estimates. Some of these estimates rely on the validity of the Sobolev inequality, while others additionally require the Poincar\'e inequality to hold. Furthermore, this section develops the Moser-type iteration scheme, which is used to establish the uniform $L^\infty$-
  bound. Finally, in Sections \ref{dim1} and \ref{dim2}, Theorems \ref{teo1} and \ref{teo2} are proved, respectively. 

\section{Statement of main results}\label{statement}
Solutions to \eqref{problema} will be meant in the weak sense according to the following definition:

\begin{definition}\label{def21}
Let $m,p,q$ be as in \eqref{hp} and let $0\leq u_0\in{ L}^{1}_{\textrm{loc}}(M)$. We say that a non-negative function $u$ is a  weak solution to problem \eqref{problema} in the time interval $[0,T)$ if
$$
u^m\in { L}^p((0,T);{ W}^{1,p}_{\rm loc}(M)),\quad u\in { L}^{q}_{\rm loc}(M\times(0,T))  
$$
and for any $\varphi \in C_c^{\infty}(M\times[0,T])$ such that $\varphi(x,T)=0$ for any $x\in M$, $u$ satisfies the equality:
\begin{equation*} 
\begin{aligned}
-\int_0^T\int_{M} \,u\,\varphi_t\,d\mu\,dt \,+&\int_0^T\int_{M} |\nabla u^m|^{p-2}\left \langle \nabla u^m, \nabla \varphi \right \rangle\,d\mu\,dt\,= \int_0^T\int_{M} \,u^{q}\,\varphi\,d\mu\,dt \\
& +\int_{M} \,u_0(x)\,\varphi(x,0)\,d\mu.
\end{aligned}
\end{equation*}
\end{definition}
\smallskip

As we already mentioned, in the first theorem, we only assume the validity of the Sobolev inequality \eqref{S}, and we  derive an $L^{\infty}$ estimate for the solution to \eqref{problema} under suitable assumptions on the initial datum, in the range $q>m(p-1)+\frac{p}{N}$. We define
\begin{equation}\label{q0}
q_0:=[q-m(p-1)]\frac{N}{p}\,.
\end{equation}
Note that $q_0>1$ whenever $q>m(p-1)+\frac pN$.
\begin{theorem}\label{teo1}
Let $M$ be a complete, non-compact Riemannian manifold of infinite volume such that  the Sobolev inequality \eqref{S} holds. 
Let also $m,p$ be as in \eqref{hp}, and assume that
 \begin{equation*}
 q>m(p-1)+\frac{p}{N},
 \end{equation*} 
 and $s>q_0$, where $q_0$ is defined in \eqref{q0}.
Moreover, assume that the initial datum $u_0\ge0$ satisfies
\begin{equation}\label{epsilon0}
u_0\in{ L}^{s}(M)\cap L^{1+m}(M)\cap L^1(M), \quad \|u_0\|_{ L^{s}(M)}\,<\,\varepsilon_0,\quad  \|u_0\|_{ L^{1}(M)}<\,\varepsilon_0\,,
\end{equation}
where $\varepsilon_0=\varepsilon_0(q,m,p,N, \mathcal S_{p})>0$ is sufficiently small. Then problem \eqref{problema} admits a solution for any $T>0$,  in the sense of Definition \ref{def21}. Furthermore, for any $\tau>0,$ the solution $u$ satisfies $u\in L^{\infty}(M\times(\tau,+\infty))$, and there exists a constant $c>0$ such that, for all $t>0$,
\begin{equation}\label{eq21tot}
\|u(t)\|_{L^{\infty}(M)}\le c\, \left(\frac{\|u_0\|_{L^{1}(M)}^p}{t^N}\right)^{\frac{1}{N[m(p-1)-1]+p}}.
\end{equation}
\end{theorem}

To obtain global in time existence for solutions in a broader range of $q$, specifically $q>m(p-1)$, we restrict the analysis to manifolds $M$ where the Poincar\'e inequality also holds.
\begin{theorem}\label{teo2}
Let $M$ be a complete, non-compact Riemannian manifold of infinite volume such that both the Sobolev inequality \eqref{S} and the Poincaré inequality \eqref{P} hold. Let also $m,p$ be as in \eqref{hp}, and assume that 
$$
q>m(p-1)>1.
$$
Let $u_0\ge0$ be such that $u_0\in{ L}^{s}(M)\cap L^{1+m}(M)\cap { L}^{q\frac Np}(M)$, for some $s>\max\left\{q_0,1\right\}$ with $q_0$ as in \eqref{q0}. Also assume that
\begin{equation*}\label{epsilon1}
\|u_0\|_{ L^{s}(M)}\,<\,\varepsilon_1, \quad \|u_0\|_{ L^{q\frac Np}(M)}\,<\,\varepsilon_1,
\end{equation*}
for some $\varepsilon_1=\varepsilon_1(q,m,p,N,s, \mathcal{S}_{p}, \mathcal{C}_p)$ sufficiently small. Then, problem \eqref{problema} admits a solution for any $T>0$,  in the sense of Definition \ref{def21}. Moreover, for any $\tau>0,$ the solution $u$ satisfies $u\in L^{\infty}(M\times(\tau,+\infty))$ and, for any $r>s$, there exists a constant $\Gamma>0$ such that, one has, for all $t>0$,
\begin{equation}\label{eq801}
\|u(t)\|_{L^{\infty}(B_R)}\le \Gamma\, t^{-\beta_{rs}}\,\|u_0\|_{L^{s}(B_R)}^{\frac{ps}{N[m(p-1)-1]+pr}},
\end{equation}
where
\begin{equation*}\label{beta}
\beta_{r,s}:=\frac{1}{m(p-1)-1}\left(1-\frac{ps}{N[m(p-1)-1]+pr}\right)>0\,.
\end{equation*}
\end{theorem}

\section{Caccioppoli-type estimates and the iterative scheme}\label{aux}

Let $x_0\in M$ be a given reference point. For any $x\in M$, we denote by $r(x)=\operatorname{dist}(x_0,x)$ the Riemannian distance between $x$ and $x_0$. In what follows, we let $$B_R(x_0):=\{x\in M\,: r(x)<R\}$$ be the geodesic ball with centre $x_0 \in M$ and radius $R > 0$. We set $B_R\equiv B_R(x_0)$. As mentioned above, we denote by $\mu$ the Riemannian measure on $M$.
\smallskip

For any $k\in\R^+$ we define the following standard truncation function on $\R$: 
\begin{equation}\label{31}
\begin{aligned}
&T_k(s):=\begin{cases} k\quad &\text{if}\,\,\, s\ge k\,, \\ s \quad &\text{if}\,\,\, |s|< k\,, \\ -k\quad &\text{if}\,\,\, s\le -k\,.\end{cases}
\end{aligned}
\end{equation}
For every $R>0$, $k>0,$ and for some $\varepsilon>0$ small enough, we consider the problem
\begin{equation}\label{eq111}
\begin{cases}
u_t = \Delta_p u^m +\, \varepsilon \Delta u+\, T_k(u^q)\quad&\text{in}\,\, B_R\times (0,+\infty) \\
u=0 &\text{in}\,\, \partial B_R\times (0,+\infty)\\
u=u_0 &\text{in}\,\, B_R\times \{0\}, \\
\end{cases}
\end{equation}
where $u_0\in L^\infty(B_R)$, $u_0\geq 0$. We will deal with solutions $u_R$ to problem \eqref{eq111} for arbitrary fixed $R>0$. Moreover, for convenience, we will denote $u_R$ with $u$. The solution is meant in the weak sense as follows.
\begin{definition}\label{def31}
Let $m,p,q$ be as in \eqref{hp}. Let $u_0\in L^\infty(B_R)$, $u_0\geq 0$. We say that a non-negative function $u$ is a solution to problem \eqref{eq111} if
$$
u\in L^{\infty}(B_R\times(0,+\infty)), \,\,\, u^m\in {L^p}\big((0, T); W^{1,p}_{0}(B_R)\big) \quad\quad \textrm{ for any }\, T>0,
$$
\noindent and for any $T>0$, $\varphi \in C_c^{\infty}(B_R\times[0,T])$ such that $\varphi(x,T)=0$ for every $x\in B_R$, $u$ satisfies the equality:
\begin{equation*}
\begin{aligned}
-\int_0^T\int_{B_R} \,u\,\varphi_t\,d\mu\,dt +& \int_0^T\int_{B_R}  |\nabla u^m|^{p-2}\left \langle \nabla u^m, \nabla \varphi \right \rangle \,d\mu\,dt\,+\varepsilon \int_0^T\int_{B_R} \left \langle\nabla u,\nabla \varphi\right \rangle\,d\mu\,dt\\&= \int_0^T\int_{B_R} \,T_k(u^q)\,\varphi\,d\mu\,dt  +\int_{B_R} \,u_0(x)\,\varphi(x,0)\,d\mu.
\end{aligned}
\end{equation*}
\end{definition}
\medskip 
\begin{remark}

Let us observe that a solution $u$  to problem \eqref{eq111} in the sense of   Definition \ref{def31} is,  in particular,  a distributional solution of the equation in \eqref{eq111}. Hence,  a standard cut-off  and density argument allows us to check that, for any $0<t_1<t_2<T$ one has 
\begin{equation*}
\begin{aligned}
&\int_{B_R} \,u(x,t_2)\,\varphi(x,t_2)\,d\mu-\int_{t_1}^{t_2}\int_{B_R} \,u\,\varphi_t\,d\mu\,dt + \int_{t_1}^{t_2}\int_{B_R}  |\nabla u^m|^{p-2}\left \langle \nabla u^m, \nabla \varphi \right \rangle \,d\mu\,dt\, \\&+\varepsilon \int_{t_1}^{t_2}\int_{B_R} \left \langle\nabla u,\nabla \varphi\right \rangle\,d\mu\,dt= \int_{t_1}^{t_2}\int_{B_R} \,T_k(u^q)\,\varphi\,d\mu\,dt  +\int_{B_R} \,u(x,t_1)\,\varphi(x,t_1)\,d\mu,
\end{aligned}
\end{equation*}
for any $\varphi \in {\rm Lip}_{\rm loc}(B_R\times [0,T])$. In partucular,  since $u_0$ is bounded and $T_k(u^{q})$ is a bounded and Lipschitz function,  thanks to elliptic regularization,  solutions of problem \eqref{eq111} in the sense of   Definition \ref{def31} are smooth (see for instance \cite{LSU}) and  therefore $f(u)\xi$, for some Lipschitz functions $f (s)$ and  $\xi(t)$, can be taken as test function in order to get, after straightforward manipulations:
\begin{equation}\label{intbyparts}
\begin{aligned}
&\int_{B_R} \,F(u(x,t_2))\xi(t_2)\,d\mu -\int_{t_1}^{t_2}\int_{B_R} \,uf(u)\,\xi_t\,d\mu\,dt + \int_{t_1}^{t_2}\int_{B_R}  |\nabla u^m|^{p-2}\left \langle \nabla u^m, \nabla f(u) \right \rangle \,d\mu\,dt\, \\&+\varepsilon \int_{t_1}^{t_2}\int_{B_R} \left \langle\nabla u,\nabla  f(u)\right \rangle\,d\mu\,dt= \int_{t_1}^{t_2}\int_{B_R} \,T_k(u^q)\,f(u)\,d\mu\,dt  +\int_{B_R} \,F(u(x,t_1))\,\xi(t_1)\,d\mu,
\end{aligned}
\end{equation}
where $F(s)=\int_0^s f(\sigma)d\sigma$.

\end{remark}

\bigskip

For any $s\in \R$, we also define 
\begin{equation}\label{31b}
G_k(s):=s-T_k(s), 
\end{equation}
where $T_k(s)$ has been defined in \eqref{31}. Let $a_1>a_2>0$ and $t>\tau_1>\tau_2>0$. We consider, for any $i\in\mathbb N\cup\{0\}$, the sequences
\begin{equation}\label{eq32}
\begin{aligned}
&k_i:=a_2+(a_1-a_2)2^{-i}\,;\\
&\theta_i:=\tau_2+(\tau_1-\tau_2)2^{-i}\,;
\end{aligned}
\end{equation}
and the cylinders 
\begin{equation}\label{eq32b}
U_i:=B_R\times(\theta_i,t).
\end{equation}

\begin{figure}{}
	\begin{tikzpicture}[scale=1.3]
		\tdplotsetmaincoords{70}{110}
		\tdplotsetrotatedcoords{0}{0}{0}

		\begin{scope}[tdplot_rotated_coords]
			\def\radius{1.5}  
			\def\height{6}    
			
			\draw (-0.25,1.51,0.1) -- (-0.25,1.51,\height+0.1);
			\draw (0.22,-\radius-0.025,-0.09) -- (0.22,-\radius -0.025,\height-0.05);
			
			\draw (0,0,\height) circle (\radius);
			\node[left] at (-\radius+0.3,-2.1,5.5) {$t$};
			
			\draw (0,0,0) circle (\radius);
			\node[left] at (-\radius+0.3,-2.1,-0.6) {$0$};
			\node[left] at (-\radius+0.3,0,-0.3) {$B_R$};
			
			\draw[blue] (0,0,5) circle (\radius);
			\node[left] at (-\radius+0.3,-2.1,4.4){$\theta_0 = \tau_1$};
			\node[right] at (-\radius+0.3,1.15,4.8) {$U_0$};
			
			\draw[blue] (0,0,4) circle (\radius);
			\node[left, black] at 
			(-\radius+0.3,-2.1,3.4) {$\theta_{i-1}$};
			\node[right, black] at (-\radius+0.3,1.15,3.8) {$U_{i-1}$};
			
			\draw[blue] (0,0,3) circle (\radius);
			\node[left] at (-\radius+0.3,-2.1,2.3) {$\theta_i$};
			\node[right] at (-\radius+0.3,1.15,2.8) {$U_i$};
			
			\draw[blue] (0,0,1.5) circle (\radius);
			\node[left] at (-\radius+0.3,-2.1,0.8) {$\tau_2$};
			\node[right] at (-\radius+0.3,1.15,1.2) {$U_\infty$};
			
		\end{scope}
		
	\end{tikzpicture}
		\caption{{The cylinders $U_i = B_R \times (\theta_i, t)$}}
\end{figure}
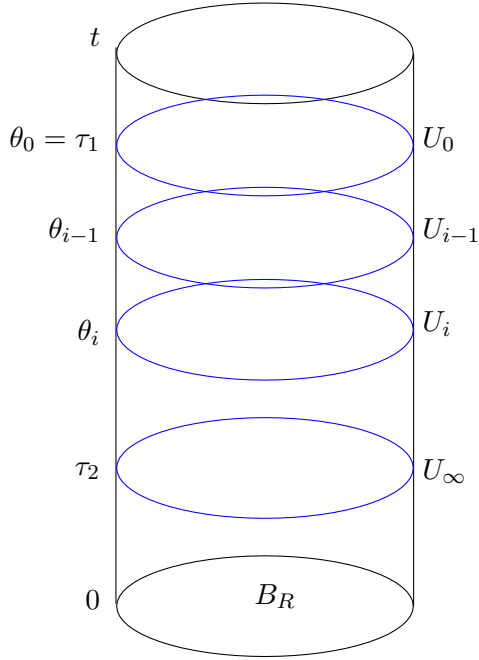

Observe that the sequence $\{\theta_i\}_{i\in \mathbb{N}}$ is monotone decreasing w.r.t. $i$. Furthermore, we define, for any $i\in\mathbb{N}$, the cut-off functions $\xi_i(\tau)$ such that
\begin{equation}\label{eq33}
\xi_i(\tau):=\begin{cases} 1\quad &\theta_{i-1}<\tau<t\\ 0\quad &0<t<\theta_i \end{cases}\quad\quad\text{and}\quad\quad|(\xi_i)_{\tau}|\,\le\, \frac{2^i}{\tau_1-\tau_2}\,.
\end{equation}
Finally, we define
\begin{equation}\label{eq34}
S(t):=\sup_{0<\tau<t}\left(\tau\|u(\tau)\|_{L^{\infty}(B_R)}^{q-1}\right).
\end{equation}
This quantity will serve as a key element in establishing that the maximal existence time $T$ of the solution is $T = +\infty$; see Lemma \ref{lemma4} and Proposition \ref{lemma5}. More precisely, in Lemma \ref{lemma3}, we derive an $L^\infty$ estimate for the solution $u$ to problem \eqref{eq111} in the time interval $t \in (0,T)$, where $T$ denotes the supremum of times such that $S(t) \le 1$. Finally, in Proposition \ref{lemma5}, we prove that this supremum is $+\infty$, thereby ensuring the existence of a bounded solution for all times, for which the estimate in Lemma \ref{lemma3} remains valid.
We can now state the following:
\begin{lemma}\label{lemma1} Let $m$ and $p$ be as in \eqref{hp}. 
Moreover, let  $i\in\mathbb{N}$, $k_i$, $\theta_i$, $U_i$ be defined in \eqref{eq32}, \eqref{eq32b} and $R>0$. Let $u$ be a non-negative solution to problem \eqref{eq111}. Then, for every $s>1$, we have that
\begin{equation*}
\begin{aligned}
\sup_{\tau_1<\tau<t}\int_{B_R}[G_{k_0}(u(x,\tau))]^s\,d\mu &+ \iint_{U_{i-1}}\left|\nabla[G_{k_i}(u)]^{\frac{m(p-1)+s-1}{p}}\right|^{p}\,d\mu d\tau \\& \le 2^i\gamma\,C_1\iint_{U_i}[G_{k_{i+1}}(u)]^s\,d\mu d\tau.
\end{aligned}
\end{equation*}
where $\gamma=\gamma(p,m,s)$ and
\begin{equation}\label{C1}C_1:=\frac{1}{\tau_1-\tau_2}+\frac{S(t)}{\tau_2}\frac{2a_1}{a_1-a_2}.\end{equation}
\end{lemma}
\begin{proof}
For any $i\in\mathbb{N}$, after a suitable regularization argument, one is allowed to take  $[G_{k_i}(u)]^{s-1}\xi_i$ ($s>1$) as a test function in \eqref{intbyparts}, that is 
\begin{equation}\label{eq36} 
\begin{aligned}
&\frac 1s\int_{B_R}[G_{k_i}(u(x,t))]^{s}\,d\mu \,-\frac 1s \iint_{U_i}[G_{k_i}(u)]^{s}(\xi_i)_{\tau}\,d\mu d\tau \\ & +\iint_{U_i}\left|\nabla u^m\right|^{p-2}\langle \nabla u^m,\,\nabla[G_{k_i}(u)^{s-1}]\rangle\xi_i\,d\mu d\tau \\ &
+\varepsilon\iint_{U_i}\langle\nabla u,\,\nabla[G_{k_i}(u)^{s-1}]\rangle\xi_i\,d\mu d\tau 
= \iint_{U_i}  T_k(u^{q})\,[G_{k_i}(u)]^{s-1}\xi_i\,d\mu d\tau\,.
\end{aligned}
\end{equation} 

Now observe that, 
\begin{equation}\label{eq38}
\begin{aligned}
\iint_{U_i}& \left|\nabla u^m\right|^{p-2}\left\langle \nabla u^m,\,\nabla[G_{k_i}(u)^{s-1}]\right\rangle\xi_i\,d\mu d\tau\\
&=m^{p-1}(s-1)\iint_{U_i} [G_{k_i}(u)]^{(m-1)(p-1)+s-2 } \left|\nabla [G_{k_i}(u)]\right|^{p}\xi_i\,d\mu d\tau\,,\\
&=m^{p-1}(s-1)\left(\frac{p}{m(p-1)+s-1}\right)^p\iint_{U_i}\left|\nabla [G_{k_i}(u)]^{\frac{ m(p-1)+s-1}{p}}\right|^{p}\xi_i\,d\mu d\tau\,.
\end{aligned}
\end{equation}
Now, combining together \eqref{eq36} and \eqref{eq38}, by using \eqref{eq33} and the fact that $T_k(u^q)\le u^q$, we can write
\begin{equation}\label{eq39}
\begin{aligned}
\frac 1s\int_{B_R}&[G_{k_i}(u(x,t))]^{s}\,d\mu + \varepsilon(s-1) \iint_{U_i}[G_{k_i}(u)]^{s-2} |\nabla u|^2\xi_i\,d\mu d\tau\\
&+ m^{p-1}(s-1)\left(\frac{p}{m(p-1)+s-1}\right)^p\iint_{U_i}\left|\nabla [G_{k_i}(u)]^{\frac{m(p-1)+s-1}{p}}\right|^{p}\xi_i\,d\mu d\tau\\
&\le \frac 1s \iint_{U_i}[G_{k_i}(u)]^{s}|(\xi_i)_{\tau}|\,d\mu d\tau+\iint_{U_i} u^{q}\,[G_{k_i}(u)]^{s-1}\xi_i\,d\mu d\tau\,\\
&\le \frac{2^i}{\tau_1-\tau_2}\iint_{U_i}[G_{k_i}(u)]^{s}\,d\mu d\tau+\iint_{U_i} u^{q}\,[G_{k_i}(u)]^{s-1}\xi_i\,d\mu d\tau.
\end{aligned}
\end{equation}
We define
\begin{equation}\label{gamma}
\gamma:=\left[\min\left\{\frac 1s,\,\, m^{p-1}(s-1)\left(\frac{p}{m(p-1)+s-1}\right)^p\right\}\right]^{-1},
\end{equation}
thus, by using \eqref{gamma} and the fact that $\varepsilon>0$, \eqref{eq39} reads
\begin{equation}\label{eq310}
\begin{aligned}
\int_{B_R}[G_{k_i}(u(x,t))]^{s}\,d\mu &+\iint_{U_i}\left|\nabla [G_{k_i}(u)]^{\frac{m(p-1)+s-1}{p}}\right|^{p}\xi_i\,d\mu d\tau\\
&\le\frac{2^i \gamma }{\tau_1-\tau_2}\iint_{U_i}[G_{k_i}(u)]^{s}\,d\mu d\tau+\gamma \iint_{U_i} u^{q}\,[G_{k_i}(u)]^{s-1}\xi_i \,d\mu d\tau.
\end{aligned}
\end{equation}
Observe that the sequence $\{k_i\}_{i\in\mathbb{N}}$ is monotone decreasing, hence
$$
G_{k_0}(u)\le ...\le  G_{k_i}(u)\le G_{k_{i+1}}(u)\le ...\le  u\quad\quad\text{for all}\,\,\,i\in\mathbb{N}.
$$
Thus \eqref{eq310} can be rewritten as
\begin{equation}\label{eq311}
\begin{aligned}
\int_{B_R}[G_{k_0}(u(x,t))]^{s}\,d\mu &+\iint_{U_{i-1}}\left|\nabla [G_{k_i}(u)]^{\frac{m(p-1)+s-1}{p}}\right|^{p}\,d\mu d\tau\\
&\le \frac{2^i\,\gamma }{\tau_1-\tau_2}\iint_{U_i}[G_{k_{i+1}}(u)]^{s}\,d\mu d\tau+\gamma \iint_{U_i} u^{q}\,[G_{k_{i+1}}(u)]^{s-1}d\mu d\tau.
\end{aligned}
\end{equation}
Now we define 
$$
I:=\gamma \iint_{U_i} u^{q-1}\,u\,[G_{k_{i+1}}(u)]^{s-1}d\mu d\tau.
$$
Observe that, for any $i\in\mathbb{N}$,
$$
\frac{u}{k_i}\chi_i\,\le\, \frac{u-k_{i+1}}{k_i-k_{i+1}}\chi_i,
$$
where $\chi_i$ is the characteristic function of $D_i:=\{(x,t)\in U_i:\,u(x,t)\ge k_i\}$. Then, by using \eqref{eq34}
\begin{equation}\label{eq312}
\begin{aligned}
I&\le\gamma \int_{\theta_i}^t \frac{1}{\tau}\tau\|u(\tau)\|_{L^{\infty}(B_R)}^{q-1}\int_{B_R}u\left[G_{k_{i+1}}(u)\right]^{s-1}\,d\mu d\tau\\
&=\gamma \int_{\theta_i}^t \frac{1}{\tau}\tau\|u(\tau)\|_{L^{\infty}(B_R)}^{q-1} \int_{B_R}k_i\frac{u}{k_i}\left[G_{k_{i+1}}(u)\right]^{s-1}\,d\mu d\tau\\
&\le\gamma \frac{k_i}{k_i-k_{i+1}}S(t)\int_{\theta_i}^{t}\frac{1}{\tau}  \int_{B_R}\left[G_{k_{i+1}}(u)\right]^{s}\,d\mu d\tau.
\end{aligned}
\end{equation}
By substituting \eqref{eq312} into \eqref{eq311} we obtain 
\begin{equation*} 
\begin{aligned}
\sup_{\tau_1<\tau<t} &\int_{B_R}[G_{k_0}(u(x,\tau))]^{s}\,d\mu +\iint_{U_{i-1}}\left|\nabla [G_{k_i}(u)]^{\frac{m(p-1)+s-1}{p}}\right|^{p}\,d\mu d\tau\\
&\le \frac{2^i\,\gamma }{\tau_1-\tau_2}\iint_{U_i}[G_{k_{i+1}}(u)]^{s}\,d\mu d\tau+\frac{k_i\,\gamma}{k_i-k_{i+1}}\frac{S(t)}{\theta_\infty} \iint_{U_i} [G_{k_{i+1}}(u)]^{s}d\mu d\tau.
\end{aligned}
\end{equation*}
Now, observe that
$$
\frac{k_i}{k_i-k_{i+1}}=\frac{2^{i+1}a_2}{a_1-a_2}+2, \quad\text{and}\quad \theta_\infty\equiv\tau_2.
$$
Consequently, by choosing $C_1$ as in \eqref{C1}, we get
\begin{equation*} 
\begin{aligned}
\sup_{\tau_1<\tau<t} \int_{B_R}[G_{k_0}(u(x,t))]^{s}\,d\mu &+\iint_{U_{i-1}}\left|\nabla [G_{k_i}(u)]^{\frac{m(p-1)+s-1}{p}}\right|^{p}\,d\mu d\tau\\
&\le 2^i\,\gamma\, C_1\iint_{U_i}[G_{k_{i+1}}(u)]^{s}\,d\mu d\tau.
\end{aligned}
\end{equation*}
\end{proof}
For any $1<r<s$ and  $\varepsilon>0$ small enough we provide the following refined version of Lemma   \ref{lemma1}
\begin{lemma}\label{lemma2}
Let $m$ and $p$ be as in \eqref{hp}. Moreover, let $1<r<s$ and assume that \eqref{S} holds. Let $k_i$, $\theta_i$, $U_i$ be defined as in \eqref{eq32}, \eqref{eq32b} and $R>0$. Let $u$ be a solution to problem \eqref{eq111}. Then, for every $i\in\mathbb{N}$ and $\varepsilon>0$, we have
\begin{equation*}
\begin{aligned}
\sup_{\tau_1<\tau<t}&\int_{B_R}[G_{k_0}(u(x,\tau))]^s\,d\mu + \iint_{U_{i-1}}\left|\nabla[G_{k_i}(u)]^{\frac{m(p-1)+s-1}{p}}\right|^{p}\,d\mu d\tau\\
&\le \varepsilon \iint_{U_i} \left|\nabla[G_{k_{i+1}}(u)]^{\frac{m(p-1)+s-1}{p}}\right|^{p} \,d\mu d\tau\\
&+ C(\varepsilon)(2^i\gamma C_1)^{\frac{N[m(p-1)+s-1-r]+p r}{N[m(p-1)-1]+p r}}(t-\tau_2)\left(\sup_{\tau_2<\tau<t}\int_{B_R}[G_{k_{\infty}}(u(x,\tau))]^r\,d\mu\right)^{\lambda},
\end{aligned}
\end{equation*}
with $C_1$ and $\gamma$ defined as in \eqref{C1} and \eqref{gamma} respectively, $\lambda=\frac{N(m(p-1)-1)+p s}{N(m(p-1)-1)+p r}$ and for some $C(\varepsilon)>0$.
\end{lemma}

\begin{remark}
Observe that, since $1<r<s$ and due to \eqref{hp}, $\lambda>1$. 
\end{remark}

\begin{proof}
Let us fix $s>1$ and $1<r<s$. We define
\begin{equation}\label{eq316}
\alpha:=r\,\frac{N[m(p-1)-1]+p s}{N[m(p-1)+s-1-r]+p r}.
\end{equation}
Observe that, due to \eqref{hp} and the choice $1<r<s$,
$$
0<\alpha<s.
$$
Now, by the H\"older inequality with exponents $\frac{p N}{N-p}\left(\frac{m(p-1)+s-1}{p(s-\alpha)}\right)$ and $\frac{N[m(p-1)+s-1]}{N[m(p-1)+\alpha-1]+p(s-\alpha)}$, we can write
\begin{equation}\label{eq317}
\begin{aligned}
\int_{B_R}[G_{k_{i+1}}(u)]^s\,d\mu&= \int_{B_R}[G_{k_{i+1}}(u)]^{s-\alpha}[G_{k_{i+1}}(u)]^\alpha\,d\mu\\
&\le\left(\int_{B_R}[G_{k_{i+1}}(u)]^{\left(\frac{m(p-1)+s-1}{p}\right)\frac{p N}{N-p}}\,d\mu\right)^{\left(\frac{p(s-\alpha)}{m(p-1)+s-1}\right)\frac{N-p}{p N}}\\
&\quad\times \left(\int_{B_R}[G_{k_{i+1}}(u)]^{\frac{\alpha\,N[m(p-1)+s-1]}{N[m(p-1)+\alpha-1]+p(s-\alpha)}}\,d\mu\right)^{\frac{N[m(p-1)+\alpha-1]+p(s-\alpha)}{N[m(p-1)+s-1]}}.
\end{aligned}
\end{equation}
Due to the definition of $\alpha$ in \eqref{eq316} and inequality \eqref{S}, inequality \eqref{eq317} becomes 
\begin{equation}\label{eq318}
\begin{aligned}
\int_{B_R}[G_{k_{i+1}}(u)]^s\,d\mu&\le\left(\frac{1}{\mathcal{S}_{p}}\left\|\nabla[G_{k_{i+1}}(u)]^{\frac{m(p-1)+s-1}{p}}\right\|_{L^p(B_R)}\right)^{\frac{p(s-\alpha)}{m(p-1)+s-1}}\\
&\quad\quad \times \left(\int_{B_R}[G_{k_{i+1}}(u)]^{r}\,d\mu\right)^{\frac\alpha r}.
\end{aligned}
\end{equation}
We multiply both sides of \eqref{eq318} by $2^i \gamma C_1$ with $\gamma$ and $C_1$ as in \eqref{C1} and \eqref{gamma} respectively  (see Lemma \ref{lemma1}). Then, by Young's inequality with exponents $\frac{m(p-1)+s-1}{s-\alpha}$ and $\frac{m(p-1)+s-1}{m(p-1)+\alpha-1}$, we get
\begin{equation}\label{eq319}
\begin{aligned}
2^i\,\gamma\,C_1\int_{B_R}[G_{k_{i+1}}(u)]^s\,d\mu&\le \varepsilon\int_{B_R}\left|\nabla[G_{k_{i+1}}(u)]^{\frac{m(p-1)+s-1}{p}}\right|^{p}\,d\mu\\
&+C(\varepsilon)(2^i\gamma\,C_1)^{\frac{m(p-1)+s-1}{m(p-1)+\alpha-1}}\left(\int_{B_R}[G_{k_{i+1}}(u)]^{r}\,d\mu\right)^{\frac\alpha r\frac{m(p-1)+s-1}{m(p-1)+\alpha-1}}.
\end{aligned}
\end{equation}
Define
$$
\lambda:=\frac\alpha r\left(\frac{m(p-1)+s-1}{m(p-1)+\alpha-1}\right)=\frac{N[m(p-1)-1]+p s}{N[m(p-1)-1]+p r}.
$$
Now, by Lemma \ref{lemma1}, one has
\begin{equation}\label{eq320}
\begin{aligned}
\sup_{\tau_1<\tau<t}\int_{B_R}[G_{k_0}(u(x,\tau))]^s\,d\mu &+ \iint_{U_{i-1}}\left|\nabla[G_{k_i}(u)]^{\frac{m(p-1)+s-1}{p}}\right|^{p}\,d\mu d\tau \\& \le 2^i\gamma\,C_1\iint_{U_i}[G_{k_{i+1}}(u)]^s\,d\mu d\tau.
\end{aligned}
\end{equation}
Furthermore, let us integrate inequality \eqref{eq319} in the time interval $\tau\in(\theta_i,t)$, then we observe that
\begin{equation}\label{eq321}
\begin{aligned}
 C(\varepsilon)&(2^i\gamma C_1)^{\frac{m(p-1)+s-1}{m(p-1)+\alpha-1}} \int_{\theta_i}^{t}\left(\int_{B_R}[G_{k_{i+1}}(u)]^{r}\,d\mu\right)^{\lambda}\,d\tau\\
 &\le C(\varepsilon)(2^i\gamma C_1)^{\frac{m(p-1)+s-1}{m(p-1)+\alpha-1}}(t-\tau_2)\left(\sup_{\tau_2<\tau<t}\int_{B_R}[G_{k_{i+1}}(u(x,\tau))]^r\,d\mu\right)^\lambda
 \end{aligned}
\end{equation}
where we have used that $\tau_2<\theta_i$ for every $i\in\mathbb{N}$.
Finally, we substitute \eqref{eq320} and \eqref{eq321} into \eqref{eq319} to get
$$
\begin{aligned}
\sup_{\tau_1<\tau<t}&\int_{B_R}[G_{k_0}(u(x,\tau))]^s\,d\mu + \iint_{U_{i-1}}\left|\nabla[G_{k_i}(u)]^{\frac{m(p-1)+s-1}{p}}\right|^{p}\,d\mu d\tau\\
&\le \varepsilon\iint_{U_i}\left|\nabla[G_{k_{i+1}}(u)]^{\frac{m(p-1)+s-1}{p}}\right|^{p}\,d\mu d\tau\\
&\quad +C(\varepsilon)(2^i\gamma C_1)^{\frac{m(p-1)+s-1}{m(p-1)+\alpha-1}}(t-\tau_2)\left(\sup_{\tau_2<\tau<t}\int_{B_R}[G_{k_{i+1}}(u(x,\tau))]^r\,d\mu\right)^\lambda.
\end{aligned}
$$
The thesis follows by noticing that, for any $i\in\mathbb{N}$
$$
G_{k_i}(u)\le G_{k_{i+1}}(u)\le \ldots \le G_{k_\infty}(u),
$$
and that
$$
\frac{m(p-1)+s-1}{m(p-1)+\alpha-1}=\frac{N[m(p-1)+s-1-r]+p r}{N[m(p-1)-1]+p r}.
$$
\end{proof}

Based on the previous Lemmas and on a Moser-type iteration procedure the following result contains a local smoothing estimate on  solutions to problem \eqref{eq111}. More precisely, with Lemmas \ref{lemma1} and \ref{lemma2} at hand, we can show that the solution $u$ to problem \eqref{eq111} with a bounded and non-negative initial datum, satisfies an $L^\infty$ estimate in the time interval $\left(\frac t2,t\right)$ depending on the $L^r$ norm of the solution itself and provided that
\(S(t)\le1\). The result is achieved by means of a Moser-type iteration procedure performed on the quantity $J_i$ defined in \eqref{eq323}.

\begin{lemma}\label{lemma3}
Let $m, p$ be as in \eqref{hp}. Let $u$ be a solution to problem \eqref{eq111} with $u_0\in L^{\infty}(B_R)$, $u_0\ge 0$ and $S(t)$ be defined as in \eqref{eq34}. Suppose that \eqref{S} holds and that, for all $t\in(0,T)$, $$S(t)\le 1.$$ Let $r\ge1$, then there exists $c=c(p,r)$ such that
\begin{equation*}
\|u(x,\tau)\|_{L^{\infty}\left(B_R\times\left(\frac t2,t\right)\right)}\,\le\,c\,t^{-\frac{N}{N[m(p-1)-1]+p r}}\left[\sup_{\frac t4<\tau<t}\int_{B_R}u^r\,d\mu\right]^{\frac{p}{N[m(p-1)-1]+p r}},
\end{equation*}
for all $t\in(0,T)$.
\end{lemma}

\begin{proof}
We define, for any $i\in\mathbb{N}$, $s>1$,  
\begin{equation}\label{eq323}
J_i:=\iint_{U_{i}}\left|\nabla[G_{k_{i+1}}(u)]^{\frac{m(p-1)+s-1}{p}}\right|^{p}\,d\mu\,dt,
\end{equation}
where $G_{k_i}$, $\{k_i\}_{i\in\mathbb{N}}$ and $U_i$ have been defined in \eqref{31b}, \eqref{eq32} and \eqref{eq32b} respectively. We  fix $1\le r<s $ and we define
\begin{equation}\label{eq324}
\beta:=\frac{N[m(p-1)+s-1-r]+p r}{N[m(p-1)-1]+p r}.
\end{equation}
Observe that   $\beta>1$ due to assumption \eqref{hp} and the choice $r<s$. By repeatedly applying   Lemma \ref{lemma2} and recalling \eqref{eq323}, we can write, for any $i\in\mathbb{N}$, the following iteration 
\begin{equation}\label{eq325}
\begin{aligned}
\sup_{\tau_1<\tau<t}&\int_{B_R}[G_{k_0}(u)]^s\,d\mu + J_0\\
&\le \varepsilon J_1+ C(\varepsilon)(2\gamma C_1)^{\beta}(t-\tau_2)\left(\sup_{\tau_2<\tau<t}\int_{B_R}[G_{k_{\infty}}(u)]^r\,d\mu\right)^{\lambda}\\
&\le \varepsilon \left\{\varepsilon J_2+ C(\varepsilon)(2^2\gamma C_1)^{\beta}(t-\tau_2)\left(\sup_{\tau_2<\tau<t}\int_{B_R}[G_{k_{\infty}}(u)]^r\,d\mu\right)^{\lambda}\right\} \\
&\quad\quad+ C(\varepsilon)(2\gamma C_1)^{\beta}(t-\tau_2)\left(\sup_{\tau_2<\tau<t}\int_{B_R}[G_{k_{\infty}}(u)]^r\,d\mu\right)^{\lambda}\\
&\le \ldots\\
&\le\varepsilon^{i} J_{i} +\sum_{j=0}^{i-1}(2^{\beta}\varepsilon)^{j}(2\gamma C_1)^\beta\,C(\varepsilon)(t-\tau_2)\left(\sup_{\tau_2<\tau<t}\int_{B_R}[G_{k_{\infty}}(u)]^r\,d\mu\right)^{\lambda},
\end{aligned}
\end{equation}
where we recall that 
\begin{equation}\label{lam2}
\lambda=\frac{N[m(p-1)-1]+p s}{N[m(p-1)-1]+p r}.
\end{equation}
We now fix $\varepsilon>0$ such that $2^\beta\varepsilon <\frac 12$. We take the limit as $i\longrightarrow+\infty$ in \eqref{eq325} thus we have
\begin{equation}\label{eq326}
\sup_{\tau_1<\tau<t}\int_{B_R}[G_{k_0}(u)]^s\,d\mu\,\le\,\tilde C(2\gamma C_1)^{\beta}(t-\tau_2)\left(\sup_{\tau_2<\tau<t}\int_{B_R}[G_{k_{\infty}}(u)]^r\,d\mu\right)^{\lambda},
\end{equation}
for some $\tilde C>1$. Observe that, due to the definition of the sequence $\{k_i\}_{i\in\mathbb{N}}$ in \eqref{eq32}, one has
$$
\begin{aligned}
&k_0=a_1\,,\quad\quad \quad \quad\quad \quad  k_{\infty}=a_2\,;\\
&G_{k_0}(u)=G_{a_1}(u)\,, \quad\quad G_{k_\infty}(u)=G_{a_2}(u)\,.
\end{aligned}
$$
For $n\in\mathbb{N}\cup\{0\}$, we consider, for some $C_0>0$ that will be fixed later, the following sequences
\begin{equation}\label{eq327}
\begin{aligned}
&t_n=\frac 12 t(1-2^{-n-1})\,;\\
&h_n=C_0(1-2^{-n-1})\,;\\
&\overline{h}_n=\frac 12(h_n+h_{n+1})\,.
\end{aligned}
\end{equation}

\begin{center}
	\begin{figure}

\begin{tikzpicture}
	\coordinate (hn) at (0,0);
	\coordinate (hnbar) at (1.5,0); 
	\coordinate (hn+1) at (3,0);
	\coordinate (C0) at (5,0);
	
	\draw (-1.5, 0) -- (6.5, 0);
	
	\filldraw[black] (hn) circle (1.5pt);
	\filldraw[black] (hnbar) circle (1.5pt);
	\filldraw[black] (hn+1) circle (1.5pt);
	\filldraw[black] (C0) circle (1.5pt);
	
	\node[below=2pt] at (hn) {$h_n$};
	\node[below=2pt] at (hnbar) {$\overline{h}_n$}; 
	
	\node[below=2pt] at (hn+1) {$h_{n+1}$};
	\node[below=2pt] at (C0) {$C_0$};

	\draw (hn) -- (hn+1);

	\draw[dashed] (-3, 0) -- (hn);
	\draw[dashed] (hn+1) -- (8, 0);
\end{tikzpicture}
	\caption{Constants approximating $C_0$}
\end{figure}
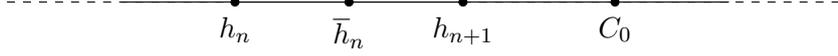
\end{center}

In \eqref{eq326} we set
\begin{equation}\label{eq328}
\tau_1=t_{n+1}\,;\quad \tau_2=t_n\,;\quad a_1=\overline h_n\,;\quad a_2=h_n\,.
\end{equation}
Then, by using \eqref{eq327} and \eqref{eq328} and the fact that, for any $t\in(0,T)$, $S(t)\le1$, we can say that there exists $C_2>1$ such that
$$
\tilde C \left(2\gamma\,C_1\right)^{\beta}\le \frac{C_2^n}{t}.
$$
Due to the latter and \eqref{eq328}, \eqref{eq326} reads
\begin{equation}\label{eq329}
\sup_{t_{n+1}<\tau<t}\int_{B_R}[G_{\overline h_n}(u)]^s\,d\mu\,\le\,C_2^{n}\,t^{-\beta+1}\left(\sup_{t_n<\tau<t}\int_{B_R}[G_{h_{n}}(u)]^r\,d\mu\right)^{\lambda},
\end{equation}
with $\beta$ and $\lambda$ as in \eqref{eq324}, \eqref{lam2}, respectively. Furthermore, observe that
\begin{equation}\label{eq330}
\int_{B_R}[G_{h_{n+1}}(u)]^r\,d\mu\,\le(h_{n+1}-\overline h_n)^{r-s}\int_{B_R}\left[G_{\overline h_n}(u)\right]^s\,d\mu.
\end{equation}
By combining together \eqref{eq329} and \eqref{eq330}, we derive the recurrent inequality
\begin{equation}\label{eq331}
\begin{aligned}
\sup_{t_{n+1}<\tau<t}\int_{B_R}[G_{h_{n+1}}(u)]^r\,d\mu\,&\le (h_{n+1}-\overline h_n)^{r-s}\sup_{t_{n+1}<\tau<t}\int_{B_R}\left[G_{\overline h_n}(u)\right]^s\,d\mu\\
&\le (h_{n+1}-\bar h_n)^{r-s}\,C_2^{n}\,t^{-\beta+1}\left(\sup_{t_n<\tau<t}\int_{B_R}[G_{h_{n}}(u)]^r\,d\mu\right)^{\lambda}.
\end{aligned}
\end{equation}
We define
$$
Y_n:=\sup_{t_n<\tau<t}\int_{B_R}[G_{h_{n}}(u)]^r\,d\mu,
$$
thus, by using \eqref{eq324}, \eqref{lam2} and \eqref{eq327}, \eqref{eq331} reads, 
\begin{equation*}
\begin{aligned}
Y_{n+1}&\le C_2^{n}\,2^{(n+3)(s-r)}\,C_0^{r-s}\,t^{-\beta+1}\,Y_n^{\lambda}\,\\
&\le c^{(s-r)n}\,C_0^{r-s}\,t^{-\beta+1}\,Y_n^{\lambda}\,,
\end{aligned}
\end{equation*}
for some $c>1$. From \cite[Chapter 2, Lemma 5.6]{LSU} it follows that
\begin{equation}\label{eq333}
Y_n\longrightarrow 0\quad\text{as}\,\,\,n\to+\infty,
\end{equation}
provided that 
\begin{equation}\label{eq334}
C_0^{r-s}\,t^{-\beta+1}\,Y_0^{\lambda-1}\le c^{r-s},
\end{equation}
Now, \eqref{eq333}, in turn, means
$$
\|u\|_{L^{\infty}\left(B_R\times\left(\frac t2,t\right)\right)}\,\le\, C_0,
$$
and, \eqref{eq334} is fulfilled if, e.g.
$$
\begin{aligned}
C_0&\le c \,t^{\frac{-\beta+1}{s-r}}\,Y_0^{\left(\lambda-1\right)\left(\frac1{s-r}\right)}\\
&\le c\,t^{-\frac{N}{N[m(p-1)-1]+p r}}\left[\sup_{\frac t4<\tau<t}\int_{B_R}u^r\,d\mu\right]^{\frac{p }{N[m(p-1)-1]+p r}},
\end{aligned}
$$ 
where we have used the definitions of $\beta$ and $\lambda$ in \eqref{eq324} and \eqref{lam2}. This concludes the proof.
\end{proof}

\subsection{An extra $L^s-L^s$ estimate}
\begin{lemma}\label{lemma41}
Assume \eqref{hp} and, besides, that  $q>m(p-1)+\frac{p}{N}.$ Assume that inequality \eqref{S} holds. Suppose that $u_0\in L^{\infty}(B_R)$, $u_0\ge0$. Let $1< s<\infty$ and assume that
\begin{equation}\label{eq41bis}
 \|u_0\|_{\textrm L^{q_0}(B_R)}\,<\,\bar\varepsilon
\end{equation}
with $\bar\varepsilon=\bar\varepsilon(s,p,m,q,\mathcal S_{p})>0$ sufficiently small and $q_0$ defined in \eqref{q0}. Let $u$ be the solution of problem \eqref{eq111} in the sense of Definition \ref{def31}, and assume that $u\in C([0,T], L^s(B_R))$ for any $s\in (1,+\infty)$, for any $T>0$. Then
\begin{equation*}
\|u(t)\|_{L^s(B_R)} \le \|u_0\|_{L^s(B_R)}\quad \textrm{ for all }\,\, t>0\,.
\end{equation*}
\end{lemma}
Note that the request $u\in C([0,T],L^s(B_R))$ for any $s\in(1,\infty)$, for any $T>0$ is not restrictive, since we will construct solutions belonging to that class.
\begin{proof}
Since $u_0$ is bounded and $T_k(u^{\s})$ is a bounded and Lipschitz function, by standard results, there exists a unique solution of problem \eqref{eq111} in the sense of Definition \ref{def31}. For any $i\in\mathbb{N}$, after a suitable regularization argument, one is allowed to take  $u^{s-1}$ ($s>1$) as a test function in \eqref{intbyparts}. For every $t\in(0, T]$ we get
\begin{equation}\label{eq43bis}
\begin{aligned}
&\frac 1s\int_{B_R}u^{s}(t)\,d\mu =\int_{B_R}\left|\nabla u^m(t)\right|^{p-2}\langle \nabla u^m(t),\,\nabla u^{s-1}(t)\rangle\,d\mu \\ &
+\varepsilon\int_{B_R}\langle\nabla u(t),\,\nabla[u^{s-1}(t)]\rangle\xi_i\,d\mu  
= \int_{B_R}  T_k(u^{q})\,u^{s-1}(t)\,d\mu \,.
\end{aligned}
\end{equation} 
Observe that, thanks to Sobolev inequality \eqref{S}, we have
\begin{equation}\label{eq43tris}
\begin{aligned}
\int_{B_R}&\left|\nabla u^m(t)\right|^{p-2}\langle \nabla u^m(t),\,\nabla u^{s-1}(t)\rangle\,d\mu\\
&=m^{p-1}(s-1)\int_{B_R} u(t)^{(m-1)(p-1)+s-2 } \left|\nabla u(t)\right|^{p}\,d\mu\\
&=m^{p-1}(s-1)\left(\frac{p}{m(p-1)+s-1}\right)^p\int_{B_R}\left|\nabla u(t)^{\frac{ m(p-1)+s-1}{p}}\right|^{p}\,d\mu \\
&\ge m^{p-1}(s-1)\left(\frac{p}{m(p-1)+s-1}\right)^p\mathcal S_{p}^p\left( \int_{B_R}  u(t)^{\frac{ m(p-1)+s-1}{p}\frac{pN}{N-p}}\,d\mu \right)^{\frac{N-p}{N}}\,.
\end{aligned}
\end{equation}
Moreover, by using the H{\"o}lder inequality with exponents $\frac{N}{N-p}$ and $\frac{N}{p}$ and since $q>q_0$ with $q_0$ as in \eqref{q0}, the last term in the right hand side of \eqref{eq43bis} becomes
\begin{equation}\label{eq43c}
\begin{aligned}
\int_{B_R} T_k(u^{q})\,u(t)^{s-1}\,dx&\le \int_{B_R} u^{q}\,u(t)^{s-1}\,dx = \int_{B_R}  u(t)^{q-m(p-1)}\,u(t)^{m(p-1)+s-1}\,dx \\
&\le \|u(t)\|^{q-m(p-1)}_{L^{(q-m(p-1))\frac{N}{p}}(B_R)} \|u(t)\|^{m(p-1)+s-1}_{L^{(m(p-1)+s-1)\frac{N}{N-p}}(B_R)}\\
&=\|u(t)\|^{q-m(p-1)}_{L^{q_0}(B_R)} \|u(t)\|^{m(p-1)+s-1}_{L^{(m(p-1)+s-1)\frac{N}{N-p}}(B_R)}\,.
\end{aligned}
\end{equation}
Combining \eqref{eq43tris} and \eqref{eq43c} we get
\begin{equation}\label{eq44bis}
\begin{aligned}
\frac{1}{q}\frac{d}{dt} \|u(t)\|^s_{L^s(B_R)}\le &-\left[m^{p-1}(s-1)\left(\frac{p}{m(p-1)+s-1}\right)^p\mathcal S_{p}^p-\|u(t)\|^{q-m(p-1)}_{L^{q_0}(B_R)} \right] \\
&\|u(t)\|^{m(p-1)+s-1}_{L^{(m(p-1)+s-1)\frac{N}{N-p}}(B_R)}\,.
\end{aligned}
\end{equation}
Take $T>0$. Observe that, due to hypotheses \eqref{eq41bis} and the known continuity in $L^{q_0}$ of the map $t\mapsto u(t)$ in $[0,T]$, there exists $t_0>0$ such that
$$
\|u(t)\|_{L^{q_0}(B_R)}\le 2\, \bar\varepsilon\,\,\,\,\,\text{for any}\,\,\,\, t\in [0,t_0]\,.
$$
Hence \eqref{eq44bis} becomes, for any $t\in (0,t_0]$,
$$
\begin{aligned}
\frac{1}{q}\frac{d}{dt} \|u(t)\|^s_{L^s(B_R)}\le &-\left[m^{p-1}(s-1)\left(\frac{p}{m(p-1)+s-1}\right)^p\mathcal S_{p}^p-(2\bar\varepsilon)^{q-m(p-1)} \right] \|u(t)\|^{m(p-1)+s-1}_{L^{(m(p-1)+s-1)\frac{N}{N-p}}(B_R)}\,.
\end{aligned}
$$
where the last inequality is obtained by using \eqref{eq41bis}. We have proved that $t\mapsto \|u(t)\|_{L^s(B_R)}$ is decreasing in time for any $t\in (0,t_0]$, thus
\begin{equation}\label{eq45bis}
\|u(t)\|_{L^s(B_R)}\le \|u_0\|_{L^s(B_R)}\quad \text{for any} \,\,\,t\in (0,t_0]\,.
\end{equation}
In particular, inequality \eqref{eq45bis} follows for the choice $s=q_0$ in view of hypothesis \eqref{eq41bis}. Hence we have
$$
\|u(t)\|_{L^{q_0}(B_R)}\le \|u_0\|_{L^{q_0}(B_R)}\,<\,\bar\varepsilon \quad \text{for any} \,\,\,\,t\in (0,t_0]\,.
$$
Now, we can repeat the same argument in the time interval $(t_0, t_1]$, with $t_1=2t_0$. This can be done due to the uniform continuity of the map $t\mapsto u(t)$ in $[0,T]$. Hence, we can write that
$$
\|u(t)\|^{q-m(p-1)}_{L^{q_0}(B_R)}\le 2\, \bar\varepsilon\,\,\,\,\,\text{for any}\,\,\, t\in (t_0,t_1]\,.
$$
Thus we get
\begin{equation*}
\|u(t)\|_{L^s(B_R)}\le \|u_0\|_{L^s(B_R)}\quad \text{for any} \,\,\,t\in (0,t_1]\,.
\end{equation*}
Iterating this procedure we obtain that $t\mapsto\|u(t)\|_{L^s(B_R)}$ is decreasing in $[0,T]$. Since $T>0$ was arbitrary, the thesis follows.
\end{proof}

\subsection{Auxiliary results when Poincar\'e inequality holds true}
In this section we furthermore assume that inequality \eqref{P} holds true and we show some local smoothing estimates.  
\begin{lemma}\label{lemma71}
Assume \eqref{hp} and, besides, that $m(p-1)>1$. Assume that inequalities \eqref{S} and \eqref{P} hold. Suppose that $u_0\in L^{\infty}(B_R)$, $u_0\ge0$. Let $1<s<\infty$ and assume that
\begin{equation}\label{eq713}
\|u_0\|_{L^{q\frac N{p}}(B_R)}<\varepsilon_1
\end{equation}
for a suitable $\varepsilon_1= \varepsilon_1(q, p, m, N, {\mathcal C}_p, \mathcal{S}_{p}, s)$ sufficiently small. Let $u$ be the solution of problem \eqref{eq111} in the sense of Definition \ref{def31}, such that in addition $u\in C([0, T); L^s(B_R))$. Then
\begin{equation}\label{eq72}
\|u(t)\|_{L^s(B_R)} \le \|u_0\|_{L^s(B_R)}\quad \textrm{ for all }\,\, t>0\,.
\end{equation}
\end{lemma}

\begin{proof}
 We use $u^{s-1}$ to test \eqref{intbyparts} and, due to the arbitrarieness of $t_1$ and $t_2$,  we get   
\begin{equation}\label{eq73}
\begin{aligned}
\frac{1}{s}\frac{d}{dt}\int_{B_R} u^{s}\,d\mu &\le-(s-1)\left(\frac{p}{m(p-1)+s-1}\right)^p\int_{B_R} \left|\nabla\left( u^{\frac{m(p-1)+s-1}{p}}\right)\right|^p \,d\mu\,\\&-\varepsilon(s-1)\int_{B_R}u^{s-2}|\nabla u|^2\,d\mu+ \int_{B_R} u^{q+s-1}\,d\mu, 
\end{aligned}
\end{equation}
where we also used that $T_k(u^q)\le u^q$.

Now we take $c_1>0$, $c_2>0$ such that $c_1+c_2=1$ so that
\begin{equation}\label{eq79}
\begin{aligned}
\int_{B_R}& \left|\nabla\left( u^{\frac{m(p-1)+s-1}{p}}\right)\right|^p d\mu \\
&= c_1\left\|\nabla\left( u^{\frac{m(p-1)+s-1}{p}}\right)\right\|_{L^p(B_R)}^p \, + c_2\left\|\nabla\left( u^{\frac{m(p-1)+s-1}{p}}\right)\right\|_{L^p(B_R)}^p\,.
\end{aligned}
\end{equation}
Take $\alpha\in (0,1)$. Thanks to \eqref{P}, \eqref{eq79} we get
\begin{equation}\label{eq74}
\begin{aligned}
\int_{B_R} &\left|\nabla\left( u^{\frac{m(p-1)+s-1}{p}}\right)\right|^p \,d\mu \ge  c_1\, \mathcal{C}_{p}^p \left\| u\right\|^{m(p-1)+s-1}_{L^{m(p-1)+s-1}(B_R)}\, + c_2\, \left\|\nabla\left( u^{\frac{m(p-1)+s-1}{p}}\right)\right\|_{L^p(B_R)}^p\\
&\ge c_1 \mathcal{C}_{p}^p \left\| u\right\|^{m(p-1)+s-1}_{L^{m(p-1)+s-1}(B_R)}\, +c_2 \left\|\nabla\left( u^{\frac{m(p-1)+s-1}{p}}\right)\right\|_{L^p(B_R)}^{p+p\alpha-p\alpha}\\
&\ge c_1 \mathcal{C}_{p}^p \left\| u\right\|^{m(p-1)+s-1}_{L^{m(p-1)+s-1}(B_R)}\, + c_2 \mathcal{C}_{p}^{p\alpha} \left\| u\right\|^{\alpha(m(p-1)+s-1)}_{L^{m(p-1)+s-1}(B_R)} \left\|\nabla\left( u^{\frac{m(p-1)+s-1}{p}}\right)\right\|_{L^p(B_R)}^{p-p\alpha}.
\end{aligned}
\end{equation}
Moreover, using the interpolation inequality, H\"{o}lder inequality and \eqref{S}, we have
\begin{equation}\label{eq75}
\begin{aligned}
\int_{B_R}  &u^{q+s-1}\,d\mu=\|u\|_{L^{q+s-1}(B_R)}^{q+s-1}\\
&\le \|u\|_{L^{m(p-1)+s-1}(B_R)}^{\theta(q+s-1)}\,\|u\|_{L^{q+m(p-1)+s-1}(B_R)}^{(1-\theta)(q+s-1)}\\
&\le \|u\|_{L^{m(p-1)+s-1}(B_R)}^{\theta(q+s-1)}\left[\|u\|_{L^{q\frac{N}{p}}(B_R)}^{q}\|u\|_{L^{(m(p-1)+s-1)\frac{N}{N-p}}(B_R)}^{m(p-1)+s-1}\right]^{\frac{(1-\theta)(q+s-1)}{q+m(p-1)+s-1}}\\
&\le \|u\|_{L^{m(p-1)+s-1}(B_R)}^{\theta(q+s-1)}\|u\|_{L^{q\frac{N}{p}}(B_R)}^{(1-\theta)\frac{q(q+s-1)}{q+m(p-1)+s-1}} \left(\frac{1}{\mathcal{S}_{p}}\left\|\nabla \left(u^{\frac{m(p-1)+s-1}{p}}\right)\right\|_{L^p(B_R)}\right)^{p(1-\theta)\frac{q+s-1}{q+m(p-1)+s-1}}
\end{aligned}
\end{equation}
where $\theta:=\frac{m(p-1)[m(p-1)+s-1]}{q(q+s-1)}$. By plugging \eqref{eq74} and \eqref{eq75} into \eqref{eq73} we obtain

\begin{equation}\label{eq710}
\begin{aligned}
\frac{1}{s}\frac{d}{dt}&\|u(t)\|_{L^s(B_R)}^{s}
 \le-(s-1)\left(\frac{p}{m(p-1)+s-1}\right)^p c_1\, \mathcal{C}_{p}^p \left\| u\right\|^{m(p-1)+s-1}_{L^{m(p-1)+s-1}(B_R)}\, \\
& - (s-1)\left(\frac{p}{m(p-1)+s-1}\right)^p c_2\, \mathcal{C}_{p}^{p\alpha} \left\| u\right\|^{\alpha(m(p-1)+s-1)}_{L^{m(p-1)+s-1}(B_R)} \left\|\nabla\left( u^{\frac{m(p-1)+s-1}{p}}\right)\right\|_{L^p(B_R)}^{p-p\alpha} \\
& +\tilde{C}\|u\|_{L^{m(p-1)+s-1}(B_R)}^{\theta(q+s-1)}\,\|u\|_{L^{q\frac{N}{p}}(B_R)}^{(1-\theta)\frac{q(q+s-1)}{q+m(p-1)+s-1}} \|\nabla \left(u^{\frac{m(p-1)+s-1}{p}}\right)\|_{L^p(B_R)}^{p(1-\theta)\frac{q+s-1}{q+m(p-1)+s-1}},
\end{aligned}
\end{equation}
where \begin{equation}\label{eq712}
\tilde{C}=\left(\frac{1}{\mathcal{S}_{p}}\right)^{p(1-\theta)\frac{q+s-1}{q+m(p-1)+s-1}}.
\end{equation}
Let us now fix $\alpha\in (0,1)$ such that
$$
p-p\alpha=p(1-\theta)\frac{q+s-1}{q+m(p-1)+s-1}.
$$
Hence, we have
\begin{equation}\label{eq76}
\alpha=\frac{m(p-1)}{q}.
\end{equation}
By substituting \eqref{eq76} into \eqref{eq710} we obtain
\begin{equation}\label{eq77}
\begin{aligned}
\frac{1}{s}\frac{d}{dt}&\|u(t)\|_{L^s(B_R)}^{s} \le -(s-1)\left(\frac{p}{m(p-1)+s-1}\right)^p c_1\, \mathcal{C}_{p}^p \left\| u\right\|^{m(p-1)+s-1}_{L^{m(p-1)+s-1}(B_R)}\\
& -  \frac{1}{\tilde C}\left\{ (s-1)\left(\frac{p}{m(p-1)+s-1}\right)^pC  - \|u\|_{L^{q\frac{N}{p}}(B_R)}^{\frac{q(q+s-1)-m(p-1)[m(p-1)+s-1]}{q+m(p-1)+s-1}}\right\}  \\
&\times\left\| u\right\|^{\alpha[m(p-1)+s-1]}_{L^{m(p-1)+s-1}(B_R)} \left\|\nabla\left( u^{\frac{m(p-1)+s-1}{p}}\right)\right\|_{L^p(B_R)}^{p-p\alpha},
\end{aligned}
\end{equation}
for some $C>0$. Observe that, due to hypotheses \eqref{eq713} and by the continuity of the solution $u(t)$, there exists $t_0>0$ such that
$$
\left\| u(t)\right\|_{L^{q\frac N{p}}(B_R)}\le 2\, \varepsilon_1\,\,\,\,\,\text{for any}\,\,\,\, t\in (0,t_0]\,.
$$
Hence, \eqref{eq77} becomes, for any $t\in (0,t_0]$
\begin{equation*}
\begin{aligned}
\frac{1}{s}\frac{d}{dt}\|u(t)\|_{L^s(B_R)}^{s} &\le  -(s-1)\left(\frac{p}{m(p-1)+s-1}\right)^p c_1 \mathcal{C}_{p}^p \left\| u\right\|^{m(p-1)+s-1}_{L^{m(p-1)+s-1}(B_R)}\, \\
& -  \frac{1}{\tilde C}\left\{ (s-1)\left(\frac{p}{m(p-1)+s-1}\right)^pC  -2 \varepsilon_1^{\frac{q(q+s-1)-m(p-1)[m(p-1)+s-1]}{q+m(p-1)+s-1}}\right\} \\& \left\| u\right\|^{\alpha(m(p-1)+s-1)}_{L^{m(p-1)+s-1}(B_R)} \left\|\nabla\left( u^{\frac{m(p-1)+s-1}{p}}\right)\right\|_{L^p(B_R)}^{p-p\alpha}\le 0,
\end{aligned}
\end{equation*}
 provided $\varepsilon_1$ is small enough. Hence we have proved that $\|u(t)\|_{L^s(B_R)}$ is decreasing in time for any $t\in (0,t_0]$, thus
\begin{equation}\label{eq78}
\|u(t)\|_{L^s(B_R)}\le \|u_0\|_{L^s(B_R)}\quad \text{for any} \,\,\,t\in (0,t_0]\,.
\end{equation}
In particular, inequality \eqref{eq78} holds for $s=q\frac N{p}$. Hence we have
$$
\|u(t)\|_{L^{q\frac N{p}}(B_R)}\le \|u_0\|_{L^{q\frac N{p}}(B_R)}\,<\,\varepsilon_1\quad \text{for any} \,\,\,\,t\in (0,t_0]\,.
$$
Now, we can repeat the same argument in the time interval $(t_0, t_1]$ with $t_1=2t_0$. This can be done due to the uniform continuity of the map $t\mapsto u(t)$ in $[0,T]$. Hence, we can write that
$$
\left\| u(t)\right\|_{L^{q\frac N{p}}(B_R)}\le 2\,\varepsilon_1\,\,\,\,\,\text{for any}\,\,\, t\in (t_0,t_1]\,.
$$
Thus we get
\begin{equation*}
\|u(t)\|_{L^{s}(B_R)}\le \|u_0\|_{L^s(B_R)}\quad \text{for any} \,\,\,t\in (0,t_1]\,.
\end{equation*}
Iterating this procedure we obtain the thesis.

\end{proof}

Let $s_0>1$ be fixed and let us define the sequence $\{s_n\}_{n\in\mathbb{N}}$ by: 
\begin{equation}\label{eq70}
s_n=m(p-1)+s_{n-1}-1, \quad \text{for all}\,\,n\in\mathbb{N}.\end{equation}
Clearly, $\{s_n\}$ is increasing and $s_n \longrightarrow +\infty$ as $n\to +\infty.$
Using a Moser type iteration procedure we prove the following $L^{s_0}-L^s$ estimate for $q>m(p-1)>1$.

\begin{proposition}\label{prop71}
Assume \eqref{hp} and, besides, that $m(p-1)>1$. Let $M$ be such that \eqref{S} and \eqref{P} hold. Suppose that $u_0\in L^{\infty}(B_R)$, $u_0\ge0$. Let $u$ be the solution of problem \eqref{eq111} in the sense of Definition \ref{def31} such that in addition $u\in C([0,T],L^s(B_R))$ for any $s\in(1,+\infty)$, for any $T>0$. Let $1< s_0\le s<+\infty$ and assume that
\begin{equation}\label{eq711}
\|u_0\|_{L^{q\frac N{p}}}(B_R)<{\varepsilon}_1
\end{equation}
for ${\varepsilon}_1={\varepsilon}_1(q,p,N,\mathcal{S}_{p}, \mathcal{C}_{p} ,s,s_0)$ sufficiently small. Then there exists $C(p,s_0,\mathcal{S}_{p},\varepsilon_1, N, s)>0$ such that
\begin{equation}\label{eq715}
\|u(t)\|_{L^s(B_R)} \le C\,t^{-\gamma_s}\|u_0\|^{\delta_s}_{L^{s_0}(B_R)}\quad \textrm{ for all }\,\, t>0\,,
\end{equation}
where
\begin{equation}\label{eq716}
\gamma_s=\frac{s_0}{m(p-1)}\left(\frac{1}{s_0}-\frac{1}{s}\right)\,,\quad \delta_s=\frac{s_0}{s}\,.
\end{equation}
\end{proposition}

\begin{proof}
Let $\{s_n\}$ be the sequence defined in \eqref{eq70}. Let $\overline n$ be the first index such that $s_{\overline n}\ge s$. Observe that $\bar n$ is well defined in view of the mentioned properties of $\{s_n\}$, see \eqref{eq70}. We start by proving a smoothing estimate from $s_0$ to $s_{\overline n}$ using again a Moser iteration technique. Afterwards, if $s_{\overline n} \equiv s$ then the proof is complete. Otherwise, if $s_{\overline n} > s$ then, by interpolation, we get the thesis.

\noindent Let $t>0$, we define
\begin{equation}\label{eq717}
r=\frac{t}{2^{\overline n}-1} , \quad t_n=(2^n-1)r\,.
\end{equation}
Observe that
$$t_0=0, \quad t_{\overline n}=t,\quad \{t_n\}\,\text{ is an increasing sequence w.r.t.}\,\,n.$$

\noindent  Now, for any $1\le n\le \overline n$, we use  \eqref{intbyparts} with  $u^{s_{n-1}-1}$  over  $B_R\times[t_{n-1},t_{n}]$ and  we get
\begin{equation*}\label{eq718}
\begin{aligned}
\frac{1}{s_{n-1}}&\left[\|u(\cdot, t_{n})\|^{s_{n-1}}_{L^{s_{n-1}}(B_R)}-\|u(\cdot, t_{n-1})\|^{q_{n-1}}_{L^{s_{n-1}}(B_R)}\right]\\&\le - (s_{n-1}-1)\left(\frac{p}{m(p-1)+s_{n-1}-1}\right)^p\int_{t_{n-1}}^{t_{n}}  \int_{B_R} \left|\nabla\left( u^{\frac{m(p-1)+s_{n-1}-1}{p}}\right)\right|^p \,d\mu\,d\tau\\
&\quad\quad+\int_{t_{n-1}}^{t_{n}}\int_{B_R}  u^q\,u^{s_{n-1}-1}\,d\mu\,d\tau,
\end{aligned}
\end{equation*} 
where we have made use of inequality
$
T_k(u^q)\,\le\,u^q.
$
Now, by arguing as in the proof of Lemma \ref{lemma71}, by using \eqref{eq79} and \eqref{eq74} with $s=s_{n-1}$, we get
\begin{equation*}\label{eq719}
\begin{aligned}
\int_{B_R} & \left|\nabla\left( u^{\frac{m(p-1)+s_{n-1}-1}{p}}\right)\right|^pd\mu \ge c_1 \mathcal{C}_{p}^p \left\| u\right\|^{m(p-1)+s_{n-1}-1}_{L^{m(p-1)+s_{n-1}-1}(B_R)}\\& + c_2 \mathcal{C}_{p}^{p\alpha} \left\| u\right\|^{\alpha(m(p-1)+s_{n-1}-1)}_{L^{m(p-1)+s_{n-1}-1}(B_R)} \left\|\nabla\left( u^{\frac{m(p-1)+s_{n-1}-1}{p}}\right)\right\|_{L^p(B_R)}^{p-p\alpha}
\end{aligned}
\end{equation*}
where $\alpha\in(0,1)$ and $c_1>0$, $c_2>0$ with $c_1+c_2=1$. Similarly, from \eqref{eq75} with $s=s_{n-1}$ we can write
\begin{equation*}\label{eq720}
\begin{aligned}
\int_{B_R}u^q u^{s_{n-1}-1}d\mu&=\|u\|_{L^{m(p-1)+s_{n-1}-1}(B_R)}^{q+s_{n-1}-1}\\
&\le \|u\|_{L^{m(p-1)+s_{n-1}-1}(B_R)}^{\theta(q+s_{n-1}-1)}\,\|u\|_{L^{q\frac{N}{p}}(B_R)}^{(1-\theta)\frac{q(q+s_{n-1}-1)}{q+m(p-1)+s_{n-1}-1}} \\
&\quad \times\left(\frac{1}{\mathcal{S}_{p}}\left\|\nabla(u^{\frac{m(p-1)+s_{n-1}-1}{p}})\right\|_{L^p(B_R)}\right)^{p(1-\theta)\frac{q+s_{n-1}-1}{q+m(p-1)+s_{n-1}-1}}
\end{aligned}
\end{equation*}
where $\theta:=\frac{(p-1)(m(p-1)+s_{n-1}-1)}{q(q+s_{n-1}-1)}$.
Now, due to assumption \eqref{eq713}, the continuity of $u$, by choosing $\tilde C$ and $\alpha$ as in \eqref{eq712} and \eqref{eq76} respectively,
we can argue as in the proof of Lemma \ref{lemma71} (see \eqref{eq77}), hence we obtain
\begin{equation}\label{eq723}
\begin{aligned}
\frac{1}{s_{n-1}}&\left[\|u(\cdot, t_{n})\|^{s_{n-1}}_{L^{s_{n-1}}(B_R)}-\|u(\cdot, t_{n-1})\|^{s_{n-1}}_{L^{s_{n-1}}(B_R)}\right]\\
&\le-(s_{n-1}-1)\left(\frac{p}{m(p-1)+s_{n-1}-1}\right)^pc_1 \mathcal{C}_{p}^p \int_{t_{n-1}}^{t_{n}}  \left\| u(\cdot, \tau)\right\|^{m(p-1)+s_{n-1}-1}_{L^{m(p-1)+s_{n-1}-1}(B_R)} d\tau \\
& -  \frac{1}{\tilde C}\left\{(s_{n-1}-1)\left(\frac{p}{m(p-1)+s_{n-1}-1}\right)^pC\,  - 2\tilde{\varepsilon_1}^{\frac{q(q+s_{n-1}-1)-(p-1)(m(p-1)+s_{n-1}-1)}{q+m(p-1)+s_{n-1}-1}}\right\}  \\
&\times \int_{t_{n-1}}^{t_n}\left\| u(\cdot,\tau)\right\|^{\alpha(m(p-1)+s_{n-1}-1)}_{L^{m(p-1)+s_{n-1}-1}(B_R)} \left\|\nabla\left( u^{\frac{m(p-1)+s_{n-1}-1}{p}}\right)(\cdot,\tau)\right\|_{L^p(B_R)}^{p-p\alpha}\,d\tau,
\end{aligned}
\end{equation}
for some $C>0$. Finally, provided $\tilde\varepsilon_1$ is small enough, \eqref{eq723} can be rewritten as
\begin{equation*}
\begin{aligned}
\frac{1}{s_{n-1}}&\left[\|u(\cdot, t_{n})\|^{s_{n-1}}_{L^{s_{n-1}}(B_R)}-\|u(\cdot, t_{n-1})\|^{s_{n-1}}_{L^{s_{n-1}}(B_R)}\right]\\
&\le-(s_{n-1}-1)\left(\frac{p}{m(p-1)+s_{n-1}-1}\right)^pc_1 \mathcal{C}_{p}^p \int_{t_{n-1}}^{t_{n}}  \left\| u(\cdot, \tau)\right\|^{m(p-1)+s_{n-1}-1}_{L^{m(p-1)+s_{n-1}-1}(B_R)} d\tau \\
\end{aligned}
\end{equation*}

\noindent We define $s_n$ as in \eqref{eq70}, so that $s_n=m(p-1)+s_{n-1}-1$. Then, in view of hypothesis \eqref{eq711}, we can apply Lemma \ref{lemma71} to the integral in the right-hand side of the latter, hence we get
\begin{equation}\label{eq724}
\begin{aligned}
\frac{1}{s_{n-1}}&\left[\|u(\cdot, t_{n})\|^{s_{n-1}}_{L^{s_{n-1}}(B_R)}-\|u(\cdot, t_{n-1})\|^{s_{n-1}}_{L^{s_{n-1}}(B_R)}\right]\\
&\le-(s_{n-1}-1)\left(\frac{p}{m(p-1)+s_{n-1}-1}\right)^pc_1 \mathcal{C}_{p}^p\left\| u(\cdot,t_n)\right\|^{s_n}_{L^{s_n}(B_R)} |t_n-t_{n-1}|.
\end{aligned}
\end{equation}
Observe that
\begin{equation}\label{eq725}
\begin{aligned}
&\|u(\cdot, t_{n})\|^{s_{n-1}}_{L^{s_{n-1}}(B_R)}\,\ge\,0,\\
&|t_n-t_{n-1}|=\frac{2^{n-1}t}{2^{\overline n}-1}.
\end{aligned}
\end{equation}
We define
\begin{equation}\label{eq726}
d_{n-1}:=\left(\frac{p}{m(p-1)+s_{n-1}-1}\right)^{-p}\frac 1{c_1\, \mathcal{C}_{p}^p}\frac{1}{s_{n-1}(s_{n-1}-1)}.
\end{equation}
By plugging \eqref{eq725} and \eqref{eq726} into \eqref{eq724}, we get
$$
\left\| u(\cdot,t_n)\right\|^{s_n}_{L^{s_n}_{\rho}(B_R)}\,\le \,\frac{2^{\bar n}-1}{2^{n-1}t}\,d_{n-1}\|u(\cdot,t_{n-1})\|^{s_{n-1}}_{L^{s_{n-1}}_{\rho}(B_R)}.
$$
The latter can be rewritten as
\begin{equation}\label{eq724b}
\left\| u(\cdot,t_n)\right\|_{L^{s_n}(B_R)}\,\le \,\left(\frac{2^{\bar n}-1}{2^{n-1}}\,d_{n-1}\right)^{\frac 1{s_n}}t^{-\frac1{s_n}}\|u(\cdot,t_{n-1})\|^{\frac{s_{n-1}}{s_n}}_{L^{s_{n-1}}(B_R)}
\end{equation}
Observe that, for any $1\le n\le \bar m$, we have
\begin{equation}\label{eq727}
\begin{aligned}
\frac{2^{\bar n}-1}{2^{n-1}}\,d_{n-1}&=\frac{2^{\bar n}-1}{2^{n-1}}\left(\frac{p}{m(p-1)+s_{n-1}-1}\right)^{-p}\frac 1{c_1\, \mathcal{C}_{p}^p}\frac{1}{s_{n-1}(s_{n-1}-1)}\\
&\le 2^{\bar n+1}\frac{1}{c_1\, \mathcal{C}_{p}^p}\left(\frac{m(p-1)+s_{n-1}-1}{p}\right)^{p}\frac{1}{s_{n-1}(s_{n-1}-1)}.
\end{aligned}
\end{equation}
Consider the function
$$
h(x):=\frac {(m(p-1)+x-1)^p}{x(x-1)}, \quad \text{for}\,\,\,s_0\le x\le s_{\overline n},\quad x\in\R.
$$
Observe that $h(x)\ge0$ for any $s_0\le x\le s_{\overline n}$. Moreover, $h$ has a maximum in the interval $s_0\le x\le s_{\overline n}$, call $\tilde{x}$ the point at which it is attained. Hence
\begin{equation}\label{eq728}
h(x)\le h(\tilde x)\quad \text{for any}\,\,\,s_0\le x\le s_{\overline n},\quad x\in\R.
\end{equation}
Due to \eqref{eq727} and \eqref{eq728}, we can say that there exists a positive constant $C$, where $C=C( \mathcal{C}_{p},\bar n,m,p,s_0)$, such that
\begin{equation}\label{eq729}
\frac{2^{\overline n}-1}{2^{n-1}}\,d_{n-1}\le C\quad \text{for all}\,\,1\le n\le \overline n.
\end{equation}
By using \eqref{eq729} and \eqref{eq724b}, we get, for any $1\le n\le \overline n$
\begin{equation}\label{eq730}
\left\| u(\cdot,t_n)\right\|_{L^{s_n}(B_R)}\,\le\, C^{\frac1{s_n}}t^{-\frac1{s_n}}\|u(\cdot,t_{n-1})\|^{\frac{s_{n-1}}{s_n}}_{L^{s_{n-1}}(B_R)}.
\end{equation}
Let us set
$$
U_n:=\left\| u(\cdot,t_n)\right\|_{L^{s_n}(B_R)}
$$
Then \eqref{eq730} becomes
$$
\begin{aligned}
U_n&\le C^{\frac1{s_n}}t^{-\frac1{s_n}}U_{n-1}^{\frac{s_{n-1}}{s_n}}\\
&\le C^{\frac1{s_n}}t^{-\frac1{s_n}}\left[C^{\frac1{s_{n-1}}}t^{-\frac1{s_{n-1}}}U_{n-2}^{\frac{s_{n-2}}{s_{n-1}}}\right]\\
&\le ...\\
&\le C^{\frac n{s_n}}t^{-\frac n{s_n}}U_0^{\frac{s_0}{s_n}}.
\end{aligned}
$$
We define
\begin{equation}\label{eq731}
\alpha_n:=\frac n{s_n},\quad \delta_n:=\frac{s_0}{s_n}.
\end{equation}
Substituting $n$ with $\bar n$ into \eqref{eq731} and in view of \eqref{eq717}, \eqref{eq730} with $n=\overline n$, we have
\begin{equation*}\label{eq732}
\left\| u(\cdot,t)\right\|_{L^{s_{\overline n}}(B_R)}\,\le\, C^{\alpha_{\overline n}}t^{-\alpha_{\overline n}}\left\| u_0\right\|_{L^{s_{0}}(B_R)}^{\delta_{\overline n}}.
\end{equation*}
Observe that if $s_{\overline n}=s$ then the thesis is proved and one has
$$
\alpha_{\overline n}=\frac1{m(p-1)-1}\left(1-\frac{s_0}{s}\right),\quad \delta_{\overline n}=\frac{s_0}{s}.
$$
Now suppose that $s<s_{\overline n}$, then in particular $s_0\le s\le s_{\overline n}$. By interpolation and Lemma \ref{lemma71} we get
\begin{equation}\label{eq733}
\begin{aligned}
\left\| u(\cdot,t)\right\|_{L^{s}(B_R)}&\le \left\| u(\cdot,t)\right\|_{L^{s_{0}}(B_R)}^{\theta}\left\| u(\cdot,t)\right\|_{L^{s_{\overline n}}(B_R)}^{1-\theta}\\
& \left\| u(\cdot,t)\right\|_{L^{s_{0}}(B_R)}^{\theta}\, C^{\alpha_{\overline n}(1-\theta)}t^{-\alpha_{\overline n}(1-\theta)}\left\| u_0\right\|_{L^{s_{0}}(B_R)}^{\delta_{\overline n}(1-\theta)}\\
&\le C^{\alpha_{\overline n}(1-\theta)}t^{-\alpha_{\overline n}(1-\theta)}\left\| u_0\right\|_{L^{q_{0}}(B_R)}^{\delta_{\overline n}(1-\theta)+\theta},
\end{aligned}
\end{equation}
where
\begin{equation}\label{eq734}
\theta=\frac{s_0}{s}\left(\frac{s_{\overline n}-s}{s_{\overline n}-s_0}\right).
\end{equation}
Combining \eqref{eq716}, \eqref{eq733} and \eqref{eq734}, we get the claim by noticing that $s$ was arbitrary fixed in $[s_0,+\infty)$.

\end{proof}

\begin{remark}
It is not possible to let $s\to+\infty$ in the bound given in Proposition \ref{prop71}. In fact, one can show that $\varepsilon \longrightarrow 0$ as $s\to\infty$. So in such limit the hypothesis on the norm of the initial datum is satisfied only when $u_0\equiv 0$.
\end{remark}

\section{Proof of Theorem \ref{teo1}}\label{dim1}

We now introduce some notation that will be useful in the next Lemma. Let $q_0$ be as in \eqref{q0} and let us define, for any $s>q_0$
\begin{equation}\label{eq41}
F(t):=\sup_{0<\tau<t}\int_{B_R}u^s(x,\tau)\,d\mu, \quad\quad F(0)=\int_{B_R}u_0^s(x)\,d\mu\,,
\end{equation}
and
\begin{equation}\label{eq42}
M(t):=\sup_{0<\tau<t}\int_{B_R}u(x,\tau)\,d\mu, \quad\quad M(0)=\int_{B_R}u_0(x)\,d\mu\,.
\end{equation}
Furthermore, we recall the definition of $S(t)$ in \eqref{eq34} and we define,
\begin{equation}\label{eq43}
\begin{aligned}
& T=\sup\{t:\,S(t)\le 1\}\,,\\
&T_F=\sup\{t:\,F(t)\le 2F(0)\}\,,\\
&T_M=\sup\{t:\,M(t)\le 2M(0)\}\,.\\
\end{aligned}
\end{equation}
Observe that, by the above definition, $S(t)\le1$ up to time $T$, therefore $T$ is the maximal time for which we can apply Lemma \ref{lemma3} and consequently have a $L^\infty$-estimate on the solution to problem \eqref{eq111}. The goal here is to show that such $T$ is actually $+\infty$ so that global existence in time follows for the approximating problem in \eqref{eq111}.
\begin{lemma}\label{lemma4}
Let $m, p$ be as in \eqref{hp} and $q>m(p-1)+\frac pN$. Let $u$ be a solution to problem \eqref{eq111} with $u_0\in L^{\infty}(B_R)$, $u_0\ge 0$, such that
\begin{equation}\label{eq427}
\|u_0\|_{L^1(B_R)}\le\varepsilon_0,\quad \|u_0\|_{L^s(B_R)}\le\varepsilon_0,
\end{equation} 
with $\varepsilon_0>0$ sufficiently small. Assume that $s>q_0$, with $q_0$ as in \eqref{q0}. Then,
$$
\min\{T,\, T_F,\,T_M\}>1\,.
$$
\end{lemma}

\begin{proof}
The proof is divided into two parts. The first part is devoted to prove the following \newline

\noindent \textit{Claim:}\,\, $\min\{T,\, T_F,\}>1$ \newline

\noindent To prove the latter claim we consider two cases: $(1): T_F<T$ and $(2): T_F\ge T$. 

\noindent $(1)$\,\,\, We  use  \eqref{intbyparts} with  $u^{s-1}$ over  $B_R\times (0,T_F)$, thus we get 
\begin{equation}\label{eq44}
\begin{aligned}
\frac 1s \left[F(T_F)-F(0)\right]&+\int_{0}^{T_F}\int_{B_R}\left|\nabla u^m\right|^{p-2}\langle \nabla u^m,\,\nabla u^{s-1}\rangle\,d\mu dt\\
&+\varepsilon\int_{0}^{T_F}\int_{B_R}\langle\nabla u,\,\nabla u^{s-1}\rangle\,d\mu dt= \int_{0}^{T_F}\int_{B_R}  T_k(u^{q})\,u^{s-1}\,d\mu dt\,,
\end{aligned}
\end{equation}
where we  exploited the definition of $F(t)$ and $T_F$ in \eqref{eq41} and \eqref{eq43}.  
Moreover,
\begin{equation}\label{eq46}
\begin{aligned}
\int_{0}^{T_F}&\int_{B_R}\left|\nabla u^m\right|^{p-2} \langle \nabla u^m,\,\nabla u^{s-1}\rangle\,d\mu dt+\varepsilon\int_{0}^{T_F}\int_{B_R}\langle\nabla u,\,\nabla u^{s-1}\rangle\,d\mu dt\\
&=m^{p-1}(s-1)\int_{0}^{T_F}\int_{B_R}u^{(m-1)(p-1)+s-2}|\nabla u|^p\,d\mu dt\\
&\quad +\varepsilon(s-1)\int_{0}^{T_F}\int_{B_R} u^{s-2}|\nabla u|^2\,d\mu dt\,;
\end{aligned}
\end{equation}
and
\begin{equation}\label{eq47}
\begin{aligned}
 \int_{0}^{T_F}\int_{B_R}  T_k(u^{q})\,u^{s-1}\,d\mu dt &\le  \int_{0}^{T_F}\|u(t)\|_{L^{\infty}(B_R)}^{q-1}\int_{B_R} u^s(x,t)\,d\mu dt\\
 & \leq F(T_F)\, \int_{0}^{T_F}\|u(t)\|_{L^{\infty}(B_R)}^{q-1}\,dt\,.
 \end{aligned}
 \end{equation}
We now substitute   \eqref{eq46}, \eqref{eq47} into \eqref{eq44}, thus we get
\begin{equation}\label{eq48}
 F(T_F)\le F(0)\,+\, s\,F(T_F)\, \int_{0}^{T_F}\|u(t)\|_{L^{\infty}(B_R)}^{q-1}\,dt\,.
\end{equation}
Let us now apply Lemma \ref{lemma3} with $r=s$, then, due to \eqref{eq41}, \eqref{eq43} and \eqref{eq427}, \eqref{eq48} reads

\begin{equation}\label{eq49}
\begin{aligned}
F(T_F)&\le F(0)\,+\, s\,F(T_F)\, c\,T_F^{1-\frac{N(q-1)}{N[m(p-1)-1]+ps}}\left[\sup_{0<t<T_F}\int_{B_R}u^s(x,t)\,d\mu,\right]^{\frac{p(q-1)}{N[m(p-1)-1]+ps}}\\
&\le F(0)\,+\, 2^{\frac{p(q-1)}{N[m(p-1)-1]+ps}}s\,F(T_F)\, c\,T_F^{1-\frac{N(q-1)}{N[m(p-1)-1]+ps}}\|u_0\|_{L^s(B_R)}^{\frac{sp(q-1)}{N[m(p-1)-1]+ps}}\\
&\le F(0)\,+\, 2^{\frac{p(q-1)}{N[m(p-1)-1]+ps}}s\,F(T_F)\, c\,T_F^{1-\frac{N(q-1)}{N[m(p-1)-1]+ps}}\varepsilon_0^{\frac{sp(q-1)}{N[m(p-1)-1]+ps}}\,.
\end{aligned}
\end{equation}
If we now assume that
$$
2^{\frac{p(q-1)}{N[m(p-1)-1]+ps}}s\,c\,T_F^{1-\frac{N(q-1)}{N[m(p-1)-1]+ps}}\varepsilon_0^{\frac{sp(q-1)}{N[m(p-1)-1]+ps}}\,\le\,\frac 13,
$$
then \eqref{eq49} reduces to
\begin{equation}\label{eq410}
F(T_F)\le F(0)\,+\, \frac 13\,F(T_F)\quad\Longrightarrow\quad F(T_F)\le \frac 32 F(0)\,.
\end{equation}
The conclusion in \eqref{eq410} yields a contradiction with respect to the definition of $T_F$ in \eqref{eq43}, in fact by definition  $F(T_F)=2F(0)$. Therefore, one can assume that
$$
2^{\frac{p(q-1)}{N[m(p-1)-1]+ps}}s\,c\,T_F^{1-\frac{N(q-1)}{N[m(p-1)-1]+ps}}\varepsilon_0^{\frac{sp(q-1)}{N[m(p-1)-1]+ps}}\,>\,\frac 13,
$$
and, by choosing $\varepsilon_0$ sufficiently small, since $s>q_0$, we conclude, for some $C>0$
$$
T_F\ge C\, \varepsilon_0^{-\frac{ps(q-1)}{N[m(p-1)-q]+ps}}\,>\,1.
$$

\noindent $(2)$\,\,\, Assume that $T<+\infty$ (otherwise the claim is obvious). Then, by applying Lemma \ref{lemma3} with $r=s$, due to \eqref{eq43} and \eqref{eq427}, we have, for some $c>0$
\begin{equation}\label{eq411}
\begin{aligned}
1=S(T)&=\sup_{0<t<T}\,\left(t\|u(t)\|_{L^{\infty}(B_R)}^{q-1}\right)\\
&\le \tilde c\,T^{1-\frac{N(q-1)}{N[m(p-1)-1]+ps}} \left(\sup_{0<t<T} \|u(t)\|_{L^s(B_R)}^{\frac{sp(q-1)}{N[m(p-1)-1]+ps}}\right) \\
&\le c\,T^{1-\frac{N(q-1)}{N[m(p-1)-1]+ps}} \|u_0\|_{L^s(B_R)}^{\frac{sp(q-1)}{N[m(p-1)-1]+ps}}\\
&\le c\,T^{1-\frac{N(q-1)}{N[m(p-1)-1]+ps}} \varepsilon_0^{\frac{sp(q-1)}{N[m(p-1)-1]+ps}}\,.
\end{aligned}
\end{equation}
Now, by the assumption $s>q_0$, if we choose $\varepsilon_0$ sufficiently small, then \eqref{eq411} yields $T>1$. The claim is proved.
\medskip

To complete the proof let us now assume that \begin{equation}\label{eq412} T_M<\min\{T,\, T_F,\},\end{equation} otherwise the proof is finished. We test \eqref{intbyparts}  over $B_R\times(0,T_M)$ with the function $\frac{u}{u+\delta}$ for some $\delta>0$ sufficiently small, in order to get 
\begin{equation}\label{eq413}
\begin{aligned}
\int_{B_R}\int_0^{u(T_M)} &\frac{z}{z+\delta}\,dzd\mu-\int_{B_R}\int_0^{u_0} \frac{z}{z+\delta}\,dzd\mu \\
&\quad+\int_{0}^{T_M}\int_{B_R}\left|\nabla u^m\right|^{p-2}\left\langle \nabla u^m,\,\nabla\left( \frac{u}{u+\delta}\right)\right\rangle\,d\mu dt\\
&\quad+\varepsilon\int_{0}^{T_M}\int_{B_R}\left\langle\nabla u,\,\nabla \left(\frac{u}{u+\delta}\right)\right\rangle\,d\mu dt\\
&= \int_{0}^{T_M}\int_{B_R}  T_k(u^{q})\,\frac{u}{u+\delta}\,d\mu dt\,.
\end{aligned}
\end{equation}
 
 Let us now consider the first term in the latter equality, by applying Tonelli's theorem we have
$$
\int_{B_R}\int_0^{u(T_M)} \frac{z}{z+\delta}\,dzd\mu=\int_0^\infty\int_{A(z)}\frac{z}{z+\delta}\,d\mu dz=\int_0^\infty\frac{z}{z+\delta}\,\mu(A(z)) dz,
$$
where
$$
A(z):=\{x\in B_R\,:\,u(x,T_M)\ge z\}\,.
$$
Let us also define
$$
D(z):=\{x\in B_R\,:\,\sup_{0<t\le T_M}u(x,t)\ge z\}\,.
$$
Then, by the layer representation of the $L^1$ norm of a function and by the definition of $T_M$ in \eqref{eq43}, we have
$$
\int_0^\infty \mu(A(z))\,dz=\|u(T_M)\|_{L^1(B_R)}=\sup_{0<t\le T_M}\|u(t)\|_{L^1(B_R)}=\int_0^\infty \mu(D(z))\,dz\,.
$$
Therefore we have showed that $\mu(A(z))=\mu(D(z))$ for almost every $z\ge0$. As a consequence we may write that
$$
\int_0^\infty\frac{z}{z+\delta}\,\mu(A(z)) dz=\int_0^\infty\frac{z}{z+\delta}\,\mu(D(z)) dz\,.
$$
Again by Tonelli's theorem we obtain
$$
\int_0^\infty\frac{z}{z+\delta}\,\mu(D(z)) dz=\int_0^\infty\int_{D(z)}\frac{z}{z+\delta}\,d\mu dz=\sup_{0<t\le T_M} \int_{B_R}\int_0^{u(t)} \frac{z}{z+\delta}\,dzd\mu\,.
$$
Thus 
\begin{equation}\label{eq414}
\int_{B_R}\int_0^{u(T_M)} \frac{z}{z+\delta}\,dzd\mu = \sup_{0<t<T_M}\int_{B_R}\int_0^{u(t)} \frac{z}{z+\delta}\,dzd\mu .
\end{equation}
 
Furthermore 
\begin{equation}\label{eq415}
\begin{aligned}
\int_{0}^{T_M}&\int_{B_R}\left|\nabla u^m\right|^{p-2} \left\langle \nabla u^m,\,\nabla \frac{u}{u+\delta}\right\rangle\,d\mu dt+\varepsilon\int_{0}^{T_M}\int_{B_R}\left\langle\nabla u,\,\nabla \frac{u}{u+\delta}\right\rangle\,d\mu dt\\
&= m^{p-1}(s-1)\int_{0}^{T_M}\int_{B_R}u^{(m-1)(p-1)}\frac{\delta}{(u+\delta)^2}|\nabla u|^p+\varepsilon\int_{0}^{T_M}\int_{B_R}\frac{\delta}{(u+\delta)^2} |\nabla u|^2\,;
\end{aligned}
\end{equation}
 
We now substitute \eqref{eq414}, \eqref{eq415} into \eqref{eq413}, thus we get
\begin{equation}\label{eq416}
\sup_{0<t<T_M}\int_{B_R}\int_0^{u(t)} \frac{z}{z+\delta}\,dzd\mu-\int_{B_R}\int_0^{u_0} \frac{z}{z+\delta}\,dzd\mu \le  \int_{0}^{T_M}\int_{B_R}  u^{q}\,\frac{u}{u+\delta}\,d\mu dt\,.
\end{equation}
By letting $\delta\to0$, \eqref{eq416} becomes 
\begin{equation}\label{eq417}
\begin{aligned}
 M(T_M)&\le M(0)\,+\, \int_{0}^{T_M}\|u(t)\|_{L^{\infty}(B_R)}^{q-1}\,\int_{B_R}u(x,t)\,d\mu\,dt\\
 &\le M(0)\,+\, \sup_{0<t<T_M}\int_{B_R}u(x,t)\,d\mu\,\int_{0}^{T_M}\|u(t)\|_{L^{\infty}(B_R)}^{q-1}\,dt\,\\
 &= M(0)\,+\, M(T_M)\, \int_{0}^{T_M}\|u(t)\|_{L^{\infty}(B_R)}^{q-1}\,dt\,.\\
 \end{aligned}
\end{equation}
Let us now apply Lemma \ref{lemma3} with $r=s$, then, due to \eqref{eq42}, \eqref{eq43}, \eqref{eq427} and \eqref{eq412}, \eqref{eq417} reads
\begin{equation*}
\begin{aligned}
M(T_M)&\le M(0)\,+\, M(T_M)\, T_M\,\|u(x,t)\|_{L^{\infty}(B_R\times(0,T_M))}^{q-1}\\
&\le M(0)\,+\, c\, M(T_M)\, T_M^{1-\frac{N(q-1)}{N[m(p-1)-1]+ps}} \left[\sup_{0<t<T_M}\int_{B_R}u^s(x,t)\,d\mu,\right]^{\frac{p(q-1)}{N[m(p-1)-1]+ps}}\\
&\le M(0)\,+\, 2\,M(T_M)\, c\,T_M^{1-\frac{N(q-1)}{N[m(p-1)-1]+ps}}\|u_0\|_{L^s(B_R)}^{\frac{sp(q-1)}{N[m(p-1)-1]+ps}}\\
&\le M(0)\,+\, 2\,M(T_M)\, c\,T_M^{1-\frac{N(q-1)}{N[m(p-1)-1]+ps}}\varepsilon_0^{\frac{sp(q-1)}{N[m(p-1)-1]+ps}}\,.
\end{aligned}
\end{equation*}
Arguing as in the case of $T_F$, see point $(1)$, assuming $\varepsilon_0$ sufficiently small, since $s>q_0$ and due to \eqref{eq43}, we obtain that $T_M>1$.
This proves the Lemma.

\end{proof}

\begin{proposition}\label{lemma5}
Let $m, p$ be as in \eqref{hp} and $q>m(p-1)+\frac pN$. Let $u$ be the solution to problem \eqref{eq111} with $u_0\in{ L}^{\infty}(B_R)$, $u_0\ge0$. Let $q_0$ be defined in \eqref{q0} and $s>q_0$. Assume that
\begin{equation}\label{eq425}
\|u_0\|_{L^1(B_R)}\le\varepsilon_0,\quad \|u_0\|_{L^s(B_R)}\le\varepsilon_0,
\end{equation}
for $\varepsilon_0>0$ sufficiently small. Let $T$ be as in \eqref{eq43}. Then
\begin{itemize}
\item[(i)] $T=+\infty.$
\item[(ii)] There exists $c>0$ such that
\begin{equation}\label{eq424}
\|u(x,t)\|_{L^{\infty}(B_R)}\,\le\,c\,t^{-\frac{N}{N[m(p-1)-1]+p}}\|u_0\|_{L^1(B_R)}^{\frac{p}{N[m(p-1)-1]+p}}, \quad \text{for any}\,\,\,t\in(0,+\infty).
\end{equation}
\end{itemize}
\end{proposition}

\begin{proof}
We first prove the following claim that will be useful later. \newline

\noindent \textit{Claim:}\,\, \begin{equation}\label{eq419b} T_M\ge T\,.\end{equation}

\noindent To prove the claim, we assume by contradiction that $T_M<T$. Then, by arguing as in proof of Lemma \ref{lemma4}, from inequality \eqref{eq417}, we get
\begin{equation}\label{eq419}
\begin{aligned}
M(T_M)&\le M(0)\,+\, M(T_M)\, \int_{0}^{T_M}\|u(t)\|_{L^{\infty}(B_R)}^{q-1}\,dt\,\\
&\le M(0)\,+\, M(T_M)\,\left[ \int_{0}^{1}\|u(t)\|_{L^{\infty}(B_R)}^{q-1}\,dt\,+ \int_{1}^{T_M}\|u(t)\|_{L^{\infty}(B_R)}^{q-1}\,dt\,\right]\,.
\end{aligned}
\end{equation}
By applying Lemma \ref{lemma3} with $r=s$, since $s>q_0$ and due to \eqref{eq425} we can write
\begin{equation}\label{eq420}
\int_{0}^{1}\|u(t)\|_{L^{\infty}(B_R)}^{q-1}\,dt\,\le \, c\,\varepsilon_0^{\frac{psq}{N[m(p-1)-1]+ps}}.
\end{equation}
Furthermore, by applying Lemma \ref{lemma3} with $r=1$, since $q>m(p-1)-1+\frac pN$ and due to \eqref{eq425} we can write
\begin{equation}\label{eq421}
\int_{1}^{T_M}\|u(t)\|_{L^{\infty}(B_R)}^{q-1}\,dt\,\le \, c\,\varepsilon_0^{\frac{pq}{N[m(p-1)-1]+p}}.
\end{equation}
We now substitute \eqref{eq420} and \eqref{eq421} into \eqref{eq419}, thus we obtain
\begin{equation}\label{eq422}
M(T_M)\le M(0)\,+\, M(T_M)\,\left[  c\,\varepsilon_0^{\frac{psq}{N[m(p-1)-1]+ps}}+ c\,\varepsilon_0^{\frac{pq}{N[m(p-1)-1]+p}} \right]\,.
\end{equation}
If we chose $\varepsilon_0$ sufficiently small to have
$$
c\,\varepsilon_0^{\frac{psq}{N[m(p-1)-1]+ps}}+ c\,\varepsilon_0^{\frac{pq}{N[m(p-1)-1]+p}}\,\le\,\frac13,
$$
then \eqref{eq422} yields
$$
M(T_M)\le \frac 32 M(0)\,,
$$
which contradicts the definition of $T_M$ in \eqref{eq43}. This proves the claim.
\medskip

By contradiction, we assume that $T<+\infty$. Then, by Lemma \ref{lemma4} and  Lemma \ref{lemma3} applied both  with $r=s$ and with $r=1$, assumption \eqref{eq425} and the choices $s>q_0$, $q>m(p-1)+\frac pN$, we get
$$
\begin{aligned}
1=S(T)&=\sup_{0<t<T}\,\left(t\|u(t)\|_{L^{\infty}(B_R)}^{q-1}\right)\\
&=\sup_{0<t<1}\,\left(t\|u(t)\|_{L^{\infty}(B_R)}^{q-1}\right)+\sup_{1<t<T}\,\left(t\|u(t)\|_{L^{\infty}(B_R)}^{q-1}\right)\\
&\le c\left(\sup_{0<t<1} t^{1-\frac{N(q-1)}{N[m(p-1)-1]+ps}}\|u(t)\|_{L^s(B_R)}^{\frac{sp(q-1)}{N[m(p-1)-1]+ps}}\right) +  c \left(\sup_{1<t<T} t^{1-\frac{N(q-1)}{N[m(p-1)-1]+p}}\|u(t)\|_{L^1(B_R)}^{\frac{p(q-1)}{N[m(p-1)-1]+p}}\right) \\
&\le c\,\|u_0\|_{L^s(B_R)}^{\frac{sp(q-1)}{N[m(p-1)-1]+ps}}+ \,c\,\|u_0\|_{L^1(B_R)}^{\frac{p(q-1)}{N[m(p-1)-1]+p}}\\
&\le c\, \varepsilon_0^{\frac{sp(q-1)}{N[m(p-1)-1]+ps}}+ c\, \varepsilon_0^{\frac{p(q-1)}{N[m(p-1)-1]+p}}\,.
\end{aligned}
$$
Provided that $\varepsilon_0$ is sufficiently small, the latter yields a contradiction, i.e. $1=S(T)<1$. Thus $T=+\infty$.
\medskip

By Lemma \ref{lemma3} with $r= 1$, we get
\begin{equation*}
\begin{aligned}
\|u(t)\|_{L^{\infty}(B_R)}&\le \|u(x,\tau)\|_{L^{\infty}\left(B_R\times\left(\frac t2,t\right)\right)}\,\\
&\le\,c\,t^{-\frac{N}{N[m(p-1)-1]+p}}\left[\sup_{\frac t4<\tau<t}\int_{B_R}u\,d\mu\right]^{\frac{p}{N[m(p-1)-1]+p}},
\end{aligned}
\end{equation*}
for all $t\in(0,T)$. Now, due to \eqref{eq419b}, the latter reduces to
\begin{equation}\label{eq423}
\|u(t)\|_{L^{\infty}(B_R)}\le\,c\,t^{-\frac{N}{N[m(p-1)-1]+p}}\|u_0\|_{L^1(B_R)}^{\frac{p}{N[m(p-1)-1]+p}}\,,
\end{equation}
for all $t\in(0,T)$. By combining \eqref{eq423} with the fact that $T=+\infty$, we get the thesis.
\end{proof}

\begin{proof}[Proof of Theorem \ref{teo1}]
Let $\{u_{0,h}\}_{h\ge 0}$ be a sequence of functions such that
\begin{equation*}
\begin{aligned}
&(a)\,\,u_{0,h}\in L^{\infty}({ M})\cap C_c^{\infty}(M) \,\,\,\text{for all} \,\,h\ge 0, \\
&(b)\,\,u_{0,h}\ge 0 \,\,\,\text{for all} \,\,h\ge 0, \\
&(c)\,\,u_{0, h_1}\leq u_{0, h_2}\,\,\,\text{for any } h_1<h_2,  \\
&(d)\,\,u_{0,h}\longrightarrow u_0 \,\,\, \text{in}\,\, L^{s}(M)\cap L^{1+m}(M) \cap L^1(M)\quad \textrm{ as }\, h\to +\infty\,.\\
\end{aligned}
\end{equation*}

Observe that, due to assumptions $(c)$ and $(d)$, $u_{0,h}$ satisfies \eqref{epsilon0}. For any $R>0$, $k>0$, $h>0$, consider the problem
\begin{equation}\label{5}
\begin{cases}
u_t= \Div\left(|\nabla u^m|^{p-2}\nabla u^m\right) +T_k(u^{q}) &\text{in}\,\, B_R\times (0,+\infty)\\
u=0& \text{in}\,\, \partial B_R\times (0,\infty)\\
u=u_{0,h} &\text{in}\,\, B_R\times \{0\}\,. \\
\end{cases}
\end{equation}
From standard results it follows that problem \eqref{5} has a solution $u_{h,k}^R$ in the sense of Definition \ref{def21} for any $T>0$. 
\noindent In view of Proposition \ref{lemma5}, the solution $u_{h,k}^R$ to problem \eqref{5} satisfies estimate \eqref{eq424} for any $t\in(0,+\infty)$, uniformly w.r.t. $R$, $k$ and $h$, i.e
\begin{equation}\label{estimatedimteo}
\|u_{h,k}^R(t)\|_{L^{\infty}(B_R)}\,\le\,c\,t^{-\frac{N}{N[m(p-1)-1]+p}}\|u_{0,h}\|_{L^1(B_R)}^{\frac{p}{N[m(p-1)-1]+p}}.
\end{equation}
We have also showed, through Proposition \ref{lemma5}, \eqref{eq42} and Lemma \ref{lemma3} (take $r=s$), that, for every $t\in(0,+\infty)$, uniformly w.r.t. $R$, $k$ and $h$, the solution satisfies
\begin{equation}\label{estimatedimteo0}
\begin{aligned}
&\|u_{h,k}^R(t)\|_{L^1(B_R)}\le \|u_{0,h}\|_{L^1(B_R)};\\
&\|u_{h,k}^R(t)\|_{L^{\infty}(B_R)}\,\le\,c\,t^{-\frac{N}{N[m(p-1)-1]+p s}}\|u_{0,h}\|_{L^s(B_R)}^{\frac{ps}{N[m(p-1)-1]+p s}}\,.
\end{aligned}
\end{equation}
Moreover, as a consequence of Definition \ref{def31}, for every $\varphi \in C_c^{\infty}(B_R\times[0,T])$ such that $\varphi(x,T)=0$ for every $x\in B_R$, the solution $u_{h,k}^R$ satisfies the equality:
\begin{equation}\label{estimatedimteo2}
\begin{aligned}
-\int_0^T\int_{B_R} \,u_{h,k}^R\,\varphi_t\,d\mu\,dt +& \int_0^T\int_{B_R}  |\nabla (u_{h,k}^R)^m|^{p-2}\left \langle \nabla (u_{h,k}^R)^m, \nabla \varphi \right \rangle \,d\mu\,dt\,\\
&= \int_0^T\int_{B_R} \,T_k((u_{h,k}^R)^q)\,\varphi\,d\mu\,dt  +\int_{B_R} \,u_{0,h}(x)\,\varphi(x,0)\,d\mu,
\end{aligned}
\end{equation}
 where the integrals are finite. Observe that, for any $h>0$, and $R>0$, the sequence $\{u_{h,k}^R\}_{k\ge0}$ is monotone increasing in $k$ hence it has a pointwise limit for $k\to\infty$. Let $u_h^R$ be such limit, therefore we write
 \[
 u_{h,k}^R\longrightarrow u_h^R\quad\text{as}\quad k\to\infty\,\,\,\text{pointwise}.
 \]
We can pass to the limit as $k\to+\infty$ in \eqref{estimatedimteo} and \eqref{estimatedimteo0}, since the right-hand side is independent of $k$, we get, for every $t>0$
 \begin{equation}\label{estimatedimteo3}
\|u_{h}^R(t)\|_{L^{\infty}(B_R)}\,\le\,c\,t^{-\frac{N}{N[m(p-1)-1]+p}}\|u_{0,h}\|_{L^1(B_R)}^{\frac{p}{N[m(p-1)-1]+p}};
\end{equation}
\begin{equation}\label{estimatedimteo3b}
\begin{aligned}
&\|u_{h}^R(t)\|_{L^1(B_R)}\le \|u_{0,h}\|_{L^1(B_R)};\\
&\|u_{h}^R(t)\|_{L^{\infty}(B_R)}\,\le\,c\,t^{-\frac{N}{N[m(p-1)-1]+p s}}\|u_{0,h}\|_{L^s(B_R)}^{\frac{ps}{N[m(p-1)-1]+p s}}\,.
\end{aligned}
\end{equation}
Moreover, in view of \eqref{estimatedimteo}, the right-hand side of \eqref{estimatedimteo2} is also independent of $k$. Thanks to Beppo Levi's monotone convergence theorem and standard monotonicity arguments (see for instance \cite{fan}), it is possible to compute the limit as $k\to\infty$ in the integrals of equality \eqref{estimatedimteo2} and obtain that
for every $\varphi \in C_c^{\infty}(B_R\times[0,T])$ such that $\varphi(x,T)=0$ for every $x\in B_R$, the function  $u_{h}^R$ satisfies the equality 
\begin{equation}\label{estimatedimteo4b}
\begin{aligned}
-\int_0^T\int_{B_R} \,u_{h}^R\,\varphi_t\,d\mu\,dt +& \int_0^T\int_{B_R}  |\nabla (u_{h}^R)^m|^{p-2}\left \langle \nabla (u_{h}^R)^m, \nabla \varphi \right \rangle \,d\mu\,dt\,\\
&= \int_0^T\int_{B_R} \,(u_{h}^R)^q\,\varphi\,d\mu\,dt  +\int_{B_R} \,u_{0,h}(x)\,\varphi(x,0)\,d\mu.
\end{aligned}
\end{equation}
Observe that the integrals in \eqref{estimatedimteo4b} are finite, hence $u_h^R$ is a solution in the sense of Definition \ref{def31} to problem \eqref{eq111} where $T_k(u^q)$ is replaced by $u^q$. In fact, we have, due to \eqref{estimatedimteo3b}, that $u_h^R\in L^1(B_R\times(0,T))$. Moreover, by using \eqref{estimatedimteo3b}, Lemma \ref{lemma41} and by the interpolation inequality, for $\theta=\frac{1+m}{q+m}<1$ 
\begin{equation}\label{estimatedimteo4}
\begin{aligned}
\int_0^T\int_{B_R} \,(u_{h}^R)^{q+m}\,d\mu dt&= \int_0^T\|u_{h}^R(t)\|_{L^{q+m}(B_R)}^{q+m}\,dt\\
&\le \int_0^T\left(\|u_h^R(t)\|_{L^{1+m}(B_R)}^\theta\|u_h^R(t)\|_{L^\infty(B_R)}^{1-\theta}\right)^{q+m} dt\\
&\le c\,\|u_{0,h}\|_{L^{1+m}(B_R)}\int_0^T t^{-\frac{N(q-1)}{N[m(p-1)-1]+ps}}\|u_{0,h}\|_{L^s(B_R)}^{\frac{ps(q-1)}{N[m(p-1)-1]+ps}}\,dt\\
&\le c\,\|u_{0,h}\|_{L^{1+m}(B_R)}\|u_{0,h}\|_{L^s(B_R)}^{\frac{ps(q-1)}{N[m(p-1)-1]+ps}}\int_0^T t^{-\frac{N(q-1)}{N[m(p-1)-1]+ps}}\,dt\\
&\le C\,\|u_{0,h}\|_{L^{1+m}(B_R)}\|u_{0,h}\|_{L^s(B_R)}^{\frac{ps(q-1)}{N[m(p-1)-1]+ps}} \,,
\end{aligned}
\end{equation}
for some finite $C>0$ independent of $h,R$. In fact, since, by the assumption, $s>q_0$ then one has $\frac{N(q-1)}{N[m(p-1)-1]+ps}<1$ and consequently
$$
\int_0^T t^{-\frac{N(q-1)}{N[m(p-1)-1]+ps}}\,dt<+\infty .
$$ Therefore $u_h^R\in L^{q+m}(B_R\times(0,T))$.  Let us now observe that, for any $h>0$, the sequence $\{u_{h}^R\}_{R\ge0}$ is monotone increasing in $R$ hence it has a pointwise limit for $R\to\infty$. Let $u_h$ be such limit therefore we write
 \[
 u_{h}^R\longrightarrow u_h\quad\text{as}\quad R\to\infty\,\,\,\text{pointwise}.
 \]
 Since $u_0\in L^1(M)\cap  L^s(M)$, there exists $c$ independent of $h$ and $R$ such that
$$
 \|u_{0,h}\|_{L^1(B_R)}\le c\,,\quad \quad \|u_{0,h}\|_{L^s(B_R)}\le c\quad\text{for all}\,\,\,h,R>0\,.
$$
Therefore the estimates in \eqref{estimatedimteo3}, \eqref{estimatedimteo3b} and \eqref{estimatedimteo4} do not depend on $R$, we can take the limit as $R\to\infty$, by Fatou's Lemma and due to the pointwise limit of $u_h^R$ we get for all $t\in(0,+\infty)$
 \begin{equation}\label{estimatedimteo5}
\|u_{h}(t)\|_{L^{\infty}(M)}\,\le\,c\,t^{-\frac{N}{N[m(p-1)-1]+p}}\|u_{0,h}\|_{L^1(M)}^{\frac{p}{N[m(p-1)-1]+p}}\,;
\end{equation}
\begin{equation}\label{estimatedimteo5b}
\begin{aligned}
&\|u_{h}(t)\|_{L^1(M)}\le \|u_{0,h}\|_{L^1(M)}\,;\\
&\|u_{h}(t)\|_{L^{\infty}(M)}\,\le\,c\,t^{-\frac{N}{N[m(p-1)-1]+p s}}\|u_{0,h}\|_{L^s(M)}^{\frac{ps}{N[m(p-1)-1]+p s}}\,;\\
&\int_0^T\int_{M} \,(u_{h})^{q+m}\,d\mu dt\le C\,\|u_{0,h}\|_{L^{1+m}(M)}\|u_{0,h}\|_{L^s(M)}^{\frac{ps(q-1)}{N[m(p-1)-1]+ps}}\,.
\end{aligned}
\end{equation}
Moreover, again by monotone convergence theorem, it is possible to compute the limit as $R\to\infty$ in the integrals of equality \eqref{estimatedimteo4b} and obtain that
for every $\varphi \in C_c^{\infty}(M\times[0,T])$ such that $\varphi(x,T)=0$ for every $x\in M$, the function $u_{h}$ satisfies the equality:
\begin{equation}\label{estimatedimteo6}
\begin{aligned}
-\int_0^T\int_{M} \,u_{h}\,\varphi_t\,d\mu\,dt +& \int_0^T\int_{M}  |\nabla (u_{h})^m|^{p-2}\left \langle \nabla (u_{h})^m, \nabla \varphi \right \rangle \,d\mu\,dt\,\\
&= \int_0^T\int_{M} \,(u_{h})^q\,\varphi\,d\mu\,dt  +\int_{M} \,u_{0,h}(x)\,\varphi(x,0)\,d\mu\,.
\end{aligned}
\end{equation}
Observe that, arguing as above, due to inequalities \eqref{estimatedimteo5} and \eqref{estimatedimteo5b}, all the integrals
in \eqref{estimatedimteo6} are finite, therefore we get that $u_h$ is a solution to problem \eqref{problema}, where we replace $u_0$ with $u_{0,h}$, in the sense of Definition \ref{def21}.
Finally, let us observe that $\{u_{0,h}\}_{h\ge 0}$ has been chosen in such a way that
$$
u_{0,h}\longrightarrow u_0 \,\,\, \text{in}\,\, L^{s}(M)\cap L^1(M)\quad \textrm{ as }\, h\to +\infty.
$$
Moreover, $\{u_{h}\}_{h\ge 0}$ is a monotone increasing function in $h$ hence it has a limit as $h\to+\infty$. We call $u$ the limit function.
Hence we can pass to the limit as $h\to+\infty$ in \eqref{estimatedimteo5}, \eqref{estimatedimteo5b} and similarly to what we have seen above we get for all $t\in(0,+\infty)$
 \begin{equation*}\label{estimatedimteo7}
\|u(t)\|_{L^{\infty}(M)}\,\le\,c\,t^{-\frac{N}{N[m(p-1)-1]+p}}\|u_{0}\|_{L^1(M)}^{\frac{p}{N[m(p-1)-1]+p}}\,;
\end{equation*}
\begin{equation*}\label{estimatedimteo7b}
\begin{aligned}
&\|u(t)\|_{L^1(M)}\le \|u_{0}\|_{L^1(M)}\,;\\
&\|u(t)\|_{L^{\infty}(M)}\,\le\,c\,t^{-\frac{N}{N[m(p-1)-1]+p s}}\|u_{0}\|_{L^s(M)}^{\frac{ps}{N[m(p-1)-1]+p s}}\,;\\
&\int_0^T\int_{M} \,u^{q+m}\,d\mu dt\le C\,\|u_{0}\|_{L^{1+m}(M)}\|u_{0}\|_{L^s(M)}^{\frac{ps(q-1)}{N[m(p-1)-1]+ps}}\,.
\end{aligned}
\end{equation*}
Moreover, again by monotone convergence, it is possible to compute the limit as $h\to+\infty$ in the integrals of equality \eqref{estimatedimteo6}. Hence we obtain that for every $\varphi \in C_c^{\infty}(M\times[0,T])$ such that $\varphi(x,T)=0$ for every $x\in M$, the function $u$ satisfies the equality:
\begin{equation*}\label{estimatedimteo8}
\begin{aligned}
-\int_0^T\int_{M} \,u\,\varphi_t\,d\mu\,dt +& \int_0^T\int_{M}  |\nabla (u)^m|^{p-2}\left \langle \nabla (u)^m, \nabla \varphi \right \rangle \,d\mu\,dt\,
\\&= \int_0^T\int_{M} \,u^q\,\varphi\,d\mu\,dt  +\int_{M} \,u_{0}(x)\,\varphi(x,0)\,d\mu\,.
\end{aligned}
\end{equation*}
Since all the integral appearing above are finite, we get that $u$ is a solution to \eqref{problema} by means of Definition \ref{def21}.

\end{proof}

\section{Proof of Theorem \ref{teo2}}\label{dim2}

\begin{lemma}\label{lemma51}
Assume $m(p-1)>1$, and $s>\max\left\{q_0,1\right\}$ with $q_0$ as in \eqref{q0}. Let $u$ be a solution to problem \eqref{eq111} with $u_0\in  L^{\infty}(B_R)$, $u_0\ge0$, such that
\begin{equation}\label{eq60}
\|u_0\|_{L^{s}(B_R)}\le\delta_1,
\end{equation}
for $\delta_1>0$ sufficiently small. Let $S(t)$ be as in \eqref{eq34} and $T$ as in \eqref{eq43}, then 
\begin{equation}\label{eq61}
T>1.
\end{equation}
\end{lemma}

\begin{proof}
By \eqref{eq34} and \eqref{eq61} one has
\begin{equation}\label{eq63}
1=S(T)=\sup_{0<t<T}\,t\|u(t)\|_{L^{\infty}(B_R)}^{q-1}.
\end{equation}
By Lemma \ref{lemma3} applied with $r=s>\max\left\{\frac Np(q-m(p-1)),1\right\}$, \eqref{eq63} gives
\begin{equation}\label{eq64}
\begin{aligned}
1=S(T)&\le \sup_{0<t<T}\,t\left\{c\,t^{-\frac{N}{N[m(p-1)-1]+p s}}\left[\sup_{\frac t4<\tau<t}\int_{B_R}u^s\,d\mu\right]^{\frac{p}{N[m(p-1)-1]+p s}}\right\}^{(q-1)}\\
&\le  \sup_{0<t<T}\, c\,t^{1-\frac{N(q-1)}{N[m(p-1)-1]+p s}}\left(\sup_{\frac t4<\tau<t}\left\|u(\tau)\right\|_{L^{s}(B_R)}^{\frac{s\,p(q-1)}{N[m(p-1)-1]+p s}}\right)\,.
\end{aligned}
\end{equation}
By applying Proposition \ref{prop71} to \eqref{eq64} and due to \eqref{eq60}, we get
\begin{equation*}\label{eq65}
\begin{aligned}
1=S(T)&\le \sup_{0<t<T}\, c\,t^{1-\frac{N(q-1)}{N[m(p-1)-1]+p s}}\left\|u_0\right\|_{L^{s}(B_R)}^{\frac{s\,p(q-1)}{N[m(p-1)-1]+p s}}\\
&\le \, c\,T^{1-\frac{N(q-1)}{N[m(p-1)-1]+p s}}\,\delta_1^{\frac{s\,p(q-1)}{N[m(p-1)-1]+p s}}\,.
\end{aligned}
\end{equation*}
The thesis follows for $\delta_1>0$ small enough.

\end{proof}

\begin{lemma}\label{lemma6}
Assume that $m(p-1)>1$ and $s>\max\left\{q_0,1\right\}$ with $q_0$ as in \eqref{q0}. Let $u$ be a solution to problem \eqref{eq111} with $u_0\in  L^{\infty}(B_R)$, $u_0\ge0$, such that
\begin{equation}\label{eq66}
\|u_0\|_{L^{s}(B_R)}\le\delta_1,\quad \|u_0\|_{L^{q\frac Np}(B_R)}\le\delta_1,
\end{equation}
for $\delta_1>0$ sufficiently small. Let $S(t)$ be as in \eqref{eq34}, then
\begin{equation}\label{eq62}
T:=\sup\{t\ge0:\,S(t)\le\,1\}=+\infty.
\end{equation}
\end{lemma}

\begin{proof}
By contradiction, assume that $$T<+\infty.$$ Then, by \eqref{eq62}, \eqref{eq34} and by Lemma \ref{lemma51} we can write,
\begin{equation}\label{eq67}
\begin{aligned}
1=S(T)&=\sup_{0<t<T}\,t\|u(t)\|_{L^{\infty}(B_R)}^{q-1}\\
&\le \sup_{0<t<1}\,t\|u(t)\|_{L^{\infty}(B_R)}^{q-1}+ \sup_{1<t<T}\,t\|u(t)\|_{L^{\infty}(B_R)}^{q-1}\\
&=:J_1+J_2\,.
\end{aligned}
\end{equation}
Now, by Lemma \ref{lemma3}, applied with $r=s$, and Lemma \ref{lemma71}, we can write
\begin{equation}\label{eq68}
\begin{aligned}
J_1&\le \,\sup_{0<t<1}\,t\left\{c\,t^{-\frac{N}{N(m(p-1)-1)+ps}}\left(\sup_{\frac t4<\tau<t}\int_{B_R}u^s\,d\mu\right)^{\frac{p}{N(m(p-1)-1)+ps}}\right\}^{(q-1)}\\
&\le  \,\sup_{0<t<1}\, c\,t^{1-\frac{N(q-1)}{N(m(p-1)-1)+ps}}\left\|u_0\right\|_{L^{s}(B_R)}^{\frac{ps(q-1)}{N(m(p-1)-1)+ps}}\,.
\end{aligned}
\end{equation}
On the other hand, for any $r>s$, by Lemma \ref{lemma3}, and Proposition \ref{prop71} with $s_0=s$, we get
\begin{equation}\label{eq69}
\begin{aligned}
J_2&\le \,\sup_{1<t<T}\,t\left\{c\,t^{-\frac{N}{N(m(p-1)-1)+pr}}\left(\sup_{\frac t4<\tau<t}\int_{B_R}u^r\,d\mu\right)^{\frac{p}{N(m(p-1)-1)+pr}}\right\}^{(q-1)}\\
&\le  \,\sup_{1<t<T}\, k\,t^{1-\frac{N(q-1)}{N(m(p-1)-1)+pr}}\sup_{\frac t4<\tau<t}\left\|u(\tau)\right\|_{L^{r}(B_R)}^{\frac{pr(q-1)}{N(m(p-1)-1)+pr}}\\
&\le  \,\sup_{1<t<T}\, c\,t^{1-\frac{N(q-1)}{N(m(p-1)-1)+pr}}\sup_{\frac t4<\tau<t}\left(C\tau^{-\frac{s}{m(p-1)-1}\left(\frac 1{s}-\frac 1r\right)}\left\|u_0\right\|_{L^{s}(B_R)}^{\frac{s}{r}}\right)^{\frac{pr(q-1)}{N(m(p-1)-1)+pr}}\\
&\le \,\sup_{1<t<T}\,\frac{C\, c}{4}\,t^{1-\frac{N(q-1)}{N(m(p-1)-1)+pr}-\frac{spr(q-1)}{(m(p-1)-1)[N(m(p-1)-1)+pr]}\left(\frac 1{s}-\frac 1r\right)}\left\|u_0\right\|_{L^{s}(B_R)}^{\frac{ps(q-1)}{N(m(p-1)-1)+pr}}\,.
\end{aligned}
\end{equation}
We substitute \eqref{eq68} and \eqref{eq69} into \eqref{eq67}, thus we get
\begin{equation}\label{eq610}
1=S(T)\le \sup_{0<t<1}\, c\,t^{\alpha}\left\|u_0\right\|_{L^{s}(B_R)}^{\frac{ps(q-1)}{N(m(p-1)-1)+ps}}+\sup_{1<t<T}\,\frac{C\, c}{4}\,t^{\beta}\left\|u_0\right\|_{L^{s}(B_R)}^{\frac{ps(q-1)}{N(m(p-1)-1)+pr}}\,,
\end{equation}
where we have defined
$$\begin{aligned}
\alpha&=1-\frac{N(q-1)}{N(m(p-1)-1)+ps}, \\
\beta&=1-\frac{N(q-1)}{N(m(p-1)-1)+pr}-\frac{spr(q-1)}{(m(p-1)-1)[N(m(p-1)-1)+pr]}\left(\frac 1{s}-\frac 1r\right)\,.\end{aligned}
$$
Since $s>\max\left\{\frac Np(q-m(p-1)),1\right\}$ and $r>s$,
$$
\alpha>0;\quad\text{and}\quad \beta<0\,.
$$
Therefore inequality \eqref{eq610}, by assumption \eqref{eq66}, reads
\begin{equation*}
1=S(T)< c\, \delta_1^{\frac{ps(q-1)}{N(m(p-1)-1)+ps}}\,+\,\frac{C\, c}{4}\delta_1^{\frac{ps(q-1)}{N(m(p-1)-1)+pr}}\,.
\end{equation*}
Provided that $\delta_1$ is sufficiently small, we get a contradiction, i.e. $1=S(T)<1$. 
\end{proof}

We conclude by proving the following 

\begin{proposition}\label{prop4}
Assume that $m(p-1)>1$ and $s>\max\left\{q_0,1\right\}$ with $q_0$ as in \eqref{q0} and $r>s$. Let $u$ be a solution to problem \eqref{eq111} with $u_0\in  L^{\infty}(B_R)$, $u_0\ge0$, such that
$$
\max\left\{\|u_0\|_{L^{s}(B_R)},\|u_0\|_{L^{q\frac Np}(B_R)}\right\}\le\varepsilon_1,
$$
 for some $\varepsilon_1=\varepsilon_1(q,m,p,N,s, \mathcal{S}_{p}, \mathcal{C}_p )$ sufficiently small. Then, for any $t\in(0,+\infty)$ and for some $\Gamma=\Gamma(q,m,p,N,r,s, \mathcal{S}_{p}, \mathcal{C}_p)>0$
\begin{equation}\label{eq800}
\|u(t)\|_{L^{\infty}(B_R)}\le \Gamma\, t^{-\frac{1}{m(p-1)-1}\left(1-\frac{ps}{N(m(p-1)-1)+pr}\right)}\,\|u_0\|_{L^{s}(B_R)}^{\frac{ps}{N(m(p-1)-1)+pr}}\,.
\end{equation}
\end{proposition}

\begin{proof}
Due to Lemma \ref{lemma6}, $$S(t)\le 1\quad \text{for all}\,\,\,t\in(0,+\infty].$$ Therefore, by Lemma \ref{lemma3} and Proposition \ref{prop71} applied with $s_0=s$, for any $r>s$, we get, for all $t\in (0,+\infty)$
\begin{equation*}
\begin{aligned}
\|u(t)\|_{L^{\infty}(B_R)}&\le \|u\|_{L^{\infty}\left(B_R\times\left(\frac t2,t\right)\right)}\,\\
&\le\, c\,t^{-\frac{N}{N(m(p-1)-1)+pr}}\left[\sup_{\frac t4<\tau<t}\|u(\tau)\|_{L^r(B_R)}^r\right]^{\frac{p}{N(m(p-1)-1)+pr}}\\
&\le\,\Gamma\,t^{-\frac{N}{N(m(p-1)-1)+pr}-\frac{s}{m(p-1)-1}\left(\frac 1{s}-\frac 1r\right)\frac{pr}{N(m(p-1)-1)+pr}}\|u_0\|_{L^{s}(B_R)}^{\frac{s}{r}\frac{pr}{N(m(p-1)-1)+pr}}\,.
\end{aligned}
\end{equation*}
The thesis follows simply observing that
\begin{equation*}
\begin{aligned}
-\frac{N}{N(m(p-1)-1)+pr}-&\frac{s}{m(p-1)-1}\left(\frac 1{s}-\frac 1r\right)\frac{pr}{N(m(p-1)-1)+pr}\\
&=-\frac 1{m(p-1)-1}\left(1-\frac{ps}{N(m(p-1)-1)+pr}\right)\,.\end{aligned}
\end{equation*}
\end{proof}

\begin{proof}[Proof of Theorem \ref{teo2}]
We proceed as in the proof of   Theorem \ref{teo1}. Let $\{u_{0,h}\}_{h\ge 0}$ be a sequence of functions such that
\begin{equation}\label{equ0h}
\begin{aligned}
&(a)\,\,u_{0,h}\in L^{\infty}(M)\cap C_c^{\infty}(M) \,\,\,\text{for all} \,\,h\ge 0, \\
&(b)\,\,u_{0,h}\ge 0 \,\,\,\text{for all} \,\,h\ge 0, \\
&(c)\,\,u_{0, h_1}\leq u_{0, h_2}\,\,\,\text{for any } h_1<h_2,  \\
&(d)\,\,u_{0,h}\longrightarrow u_0 \,\,\, \text{in}\,\, L^{s}(M)\cap L^{1+m}(M)\quad \textrm{ as }\, h\to +\infty\,.\\
\end{aligned}
\end{equation}
From standard results it follows that problem \eqref{5} has a solution $u_{h,k}^R$ in the sense of Definition \ref{def31} with $u_{0,h}$ as in \eqref{equ0h}; moreover, $u^R_{h,k}\in C\big([0, \infty); L^r(B_R)\big)$ for any $r>1$.
Due to  Propositions \ref{prop4} and \ref{prop71} and Lemmas \ref{lemma71} and \ref{lemma6}, the solution $u_{h,k}^R$ satisfies estimates \eqref{eq72}, \eqref{eq715} and \eqref{eq800} for any $t\in(0,+\infty)$, uniformly w.r.t. $R$, $k$ and $h$.
Thus, reasoning as for the proof of Theorem \ref{teo1}, we can pass to the limit as $R\to+\infty$, $k\to+\infty$ and $h\to+\infty$ and we obtain a solution $u$ to problem \eqref{problema}, which fulfills \eqref{eq801}.

\end{proof}

\section*{Acknowledgements} The first author is funded by the Deutsche Forschungsgemeinschaft (DFG, German Research Foundation) - SFB 1283/2 2021. 
The second and the third author have been  partially supported by the Gruppo Nazionale per l’Analisi Matematica, la Probabilità e le loro Applicazioni (GNAMPA) of the Istituto Nazionale di Alta Matematica (INdAM).

\section*{Data availability statement} We do not analyse or generate any datasets, because our work proceeds within a theoretical and mathematical approach. One can obtain the relevant materials from the references below.

\end{document}